\documentclass{amsart}
\usepackage{amsmath,amssymb,amsthm}

\usepackage{tikz-cd}

\usepackage{enumitem}

 \newcommand{\br}{\mathbf{r}}               
\newcommand{\bb}{\mathbf{b}}               

\newcommand{\Kfam}{\mathcal{K}}            
\newcommand{\Yeq}[1][r,s,\bd]{\mathcal{Y}^{=}_{#1}}
\newcommand{\F}{\mathbb{F}}                
\newcommand{\Z}{\mathbb{Z}}
\newcommand{\Q}{\mathbb{Q}}
\newcommand{\R}{\mathbb{R}}
\newcommand{\C}{\mathbb{C}}
\newcommand{\A}{\mathbb{A}}
\newcommand{\bP}{\mathbb{P}}
\newcommand{\WP}{\mathbb{WP}}
\newcommand{\Gm}{\mathbb{G}_m}
\newcommand{\Qbar}{\overline{\Q}}          
\newcommand{\muu}[1]{\mu_{#1}}             

\DeclareMathOperator{\id}{id}  
\DeclareMathOperator{\Jac}{Jac}
\DeclareMathOperator{\wt}{wt}              
\DeclareMathOperator{\diag}{diag}          
\DeclareMathOperator{\sgn}{sgn}            
\DeclareMathOperator{\rank}{rank}          
\DeclareMathOperator{\trdeg}{trdeg}        
\DeclareMathOperator{\Aut}{Aut}            
\DeclareMathOperator{\Gal}{Gal}            
\DeclareMathOperator{\Span}{span}          
\DeclareMathOperator{\GL}{GL}
\DeclareMathOperator{\Spec}{Spec}  

\newcommand{\detJ}[1]{\det J#1}            
\newcommand{\Jm}[1]{J#1}                   
\newcommand{\Jacuv}{\Jac_{u,v}}            
\newcommand{\jb}[2]{\{#1,#2\}}             

\newcommand{\bomega}{\boldsymbol{\omega}}

\newcommand{\Fbar}{\overline{F}}           
\newcommand{\Gbar}{\overline{G}}
\newcommand{\sigt}{\sigma_t}
\newcommand{\taut}{\tau_t}
\newcommand{\Omsig}{\Omega^{\sigma}}       
\newcommand{\Ktau}{K^{\tau}}

\newcommand{\bd}{\mathbf{d}}
\newcommand{\bw}{\mathbf{w}}
\newcommand{\bq}{\mathbf{q}}
\newcommand{\hwt}{\mathfrak{h}_{\bw}}      

\newcommand{\VA}{V_A}
\newcommand{\VB}{V_B}
\newcommand{\VL}{V_{\Lambda}}
\newcommand{\VAd}{V_A(\bd)}
\newcommand{\VBd}{V_B(\bd)}
\newcommand{\VLd}{V_{\Lambda}(\bd)}
\newcommand{\TBL}{T_{B,\Lambda}}
\newcommand{\Ypar}[1][r,s,\bd]{\mathcal{Y}_{#1}}
\newcommand{\Yparo}[1][r,s,\bd]{\mathcal{Y}^{\circ}_{#1}}
\newcommand{\Wpar}[1][r,s,\bd]{\mathcal{W}_{#1}}
\newcommand{\Gaut}{\mathcal{G}}            
\newcommand{\Ncount}{\mathcal{N}}          
\newcommand{\Ost}{\mathcal{O}}             

\usepackage[colorlinks=true,
            linkcolor=blue,
            citecolor=red,
            urlcolor=magenta]{hyperref}

\usepackage[capitalise,nameinlink]{cleveref}

\newtheorem{theorem}{Theorem}[section]
\newtheorem{lem}[theorem]{Lemma}
\newtheorem{prop}[theorem]{Proposition}
\newtheorem{cor}[theorem]{Corollary}
\newtheorem{exa}[theorem]{Example}

\theoremstyle{remark}
\newtheorem{rem}[theorem]{Remark}

\crefname{theorem}{Thm}{Theorems}
\crefname{lem}{Lem}{Lemmas}
\crefname{prop}{Prop}{Propositions}
\crefname{cor}{Cor}{Corollaries}
\crefname{rem}{Rem}{Remarks}
\crefname{section}{Sec.}{Secs.}
\crefname{equation}{Eq.}{Eqs.}

\title{Graded Keller maps and the Jacobian Conjecture}
\author{T. Shaska}
\date{\today}

\begin{document}

\begin{abstract}
We study Keller maps $\C^n\to\C^n$ that are equivariant for a linear action of $\Gm$, that is homogeneous for an integer grading of the coordinates, and ask what the weight vector alone decides. The sign pattern decides a great deal. If the weights are all of one sign, the setting of weighted projective spaces, a graded Keller map is always an automorphism; in dimension two the same holds for every sign pattern. A graded counterexample to the Jacobian Conjecture must therefore have weights of mixed sign and dimension at least three, and the map recently announced as a counterexample, equivariant for $\wt(x,y,z)=(1,-1,-2)$, realizes the smallest such case. When exactly one weight is positive and equals $1$, the Keller condition descends to the invariant quotient, where it says that the Jacobian of the quotient map vanishes along the contracted locus to the order prescribed by the weights. The descended condition is a single equation, multilinear in its arguments, so the graded counterexamples of bounded degree are exactly the rational points of an explicit multiprojective scheme over $\Z$, and the announced map is one of them. Modulo the natural $\Gm$ the parameter space is a weighted projective space with positive weights: a hyperbolic counterexample is parametrised by an elliptic moduli space, of the type in which rational points are governed by a Kummer condition at every prime \cite{sparsity}. We classify graded Keller maps by the signature of the weights, completely in the elliptic and parabolic cases, and reduce the hyperbolic case to that scheme. Over $\Q$ the image of the announced map is thin, but along the line where the stabilizer is $\muu{2}$ the Kummer condition of the positive-weight theory reappears.
\end{abstract}

\maketitle

\setcounter{tocdepth}{1}
\tableofcontents


\section{Introduction}
\label{sec-1}

A Keller map is a polynomial map $\C^n\to\C^n$ whose Jacobian determinant is a nonzero constant, and the Jacobian Conjecture asserts that every such map is a polynomial automorphism. We study Keller maps that are equivariant for a linear action of $\Gm$, that is homogeneous for a grading $\wt(x_1,\dots,x_n)=(w_1,\dots,w_n)$ with $w_i\in\Z$, and we ask what the weight vector alone decides. The answer is that it decides a great deal, and that the dividing line is the sign pattern. If the weights are all of one sign, the setting of weighted projective spaces, a graded Keller map is always an automorphism (\cref{thm:rigid}); in dimension two the same holds for every sign pattern (\cref{thm:plane}). A graded counterexample must therefore have weights of mixed sign and dimension at least three. In that remaining case the Keller condition descends to the invariant quotient, where it becomes a single equation, multilinear in its arguments (\cref{thm:master}, \cref{thm:nmaster}), and the counterexamples of bounded degree become the rational points of an explicit scheme over $\Z$ (\cref{thm:points}).

Our interest in this question comes from two directions. The first is a recent one. In July 2026, L.\ Alp\"oge announced a counterexample to the Jacobian Conjecture in dimension three, credited to the AI system Fable~\cite{alpoge}. Define
\begin{equation}\label{eq:map}
\begin{split}
F  \colon \C^3 	&	\to\C^3   \\
(x, y, z) & \to (a,b,c)
\end{split}
\end{equation}
where $a$, $b$, $c$ are given by the polynomials
\begin{equation}
\begin{aligned}
a&=(1+xy)^3z+y^2(1+xy)(4+3xy),\\
b&=y+3x(1+xy)^2z+3xy^2(4+3xy),\\
c&=2x-3x^2y-x^3z.
\end{aligned}
\end{equation}
Denote by  $\detJ {F}$   the determinant of the Jacobian matrix of $F$.  We have $\detJ {F}\equiv-2$, while $F$ is generically three-to-one. The components of $F$ are homogeneous for the grading $\wt(x,y,z)=(1,-1,-2)$, which is hyperbolic and, by the above, the smallest sign pattern and the smallest dimension in which a graded counterexample can occur. The map is the running example of this paper, and in \cref{sec-8} it becomes a $\Q$-point of the parameter scheme just mentioned. Everything we prove about the class is independent of it; the announcement is what makes the class interesting today, not what the results rest on.

That the grading exists was observed independently and earlier in \cite{dorky,sbs}, in the opposite sign convention $\wt(x,y,z)=(-1,1,2)$, as were versions of some results of \cref{sec-4} and \cref{sec:disc}; we learned of this after the present analysis was under way, and give independent proofs of everything we use.

The second direction is the one we came from. Graded vector spaces, and the weighted projective spaces attached to them, have been the setting of a series of studies: rational points and heights \cite{sparsity,bgsh}, codes \cite{gqc}, and learning architectures whose layers are graded and whose morphisms are weight-homogeneous \cite{gvs,gnn,gt}. In each the grading is the structure, and the results depend on the weights. The present paper asks the same question for polynomial maps: which properties of a graded polynomial map are already decided by its weight vector? \Cref{rem:graded-nn} records what \cref{thm:rigid} says in the architecture setting.

\section{Graded  maps}
\label{sec-2}

Let $\Gm$ be the multiplicative group, that is the affine algebraic group $\Gm=\A^1\setminus\{0\}$ with coordinate ring $\Z[t,t^{-1}]$, multiplication as the group operation, and $1$ as the identity. Its points are $\Gm(\C)=\C^{\times}=\C\setminus\{0\}$, and more generally $\Gm(k)=k^{\times}$ for any field $k$.

Let $\bw=(w_1,\dots,w_n)\in\Z^n$ be a \textbf{weight vector}. We grade the polynomial ring $\C[x_1,\dots,x_n]$ by giving $x_i$ the weight $\wt(x_i)=w_i$, so that a monomial has weight
\[
  \wt\bigl(x_1^{m_1}\cdots x_n^{m_n}\bigr)=m_1w_1+\dots+m_nw_n .
\]
A polynomial is called \textbf{weight-homogeneous} of weight $\ell$ if all its monomials have weight $\ell$; the zero polynomial is weight-homogeneous of every weight. 
 We write $\C[x]_{\bw}$ for the polynomial ring with this grading, $\C[x]_{\bw,\ell}$ for its piece of weight $\ell$, and $\A^n_{\bw}$ for the \textbf{graded affine space}, that is the affine space $\A^n$ together with the grading of its coordinate ring. It is $\A^n$ equipped with extra structure, not a new variety.
 
 All varieties are over $\C$ unless stated otherwise, and we write $\A^n=\Spec\C[x_1,\dots,x_n]$, so that $\A^n(\C)=\C^n$; the group acting is $\Gm$ base changed to $\C$. We use $\A^n$ when the grading is in play and $\C^n$ when the argument is about points, as in the topological step of \cref{thm:rigid}. The schemes of \cref{sec-8} are the exception: they are defined over $\Z$.

To any grading $\bw$ we associate the diagonal action of $\Gm$ on $\A^n$ given by
\[
  \sigt(x_1,\dots,x_n)=\bigl(t^{w_1}x_1,\dots,t^{w_n}x_n\bigr), \qquad t\in\Gm .
\]
Since every algebraic character of $\Gm$ is $t\mapsto t^{\ell}$ for a unique $\ell\in\Z$, the assignment $\bw\mapsto\sigma$ is a bijection from $\Z^n$ onto the set of algebraic $\Gm$-actions on $\A^n$ which are diagonal in the coordinates $x_1,\dots,x_n$, the inverse assignment recovering $w_i$ as the weight of $x_i$ under $\sigma$. Moreover, for $f\in\C[x]$ one has $f\in\C[x]_{\bw,\ell}$ if and only if $f\circ\sigt=t^{\ell}f$ for every $t\in\Gm$; in particular the weight-zero piece is the ring of invariants,
\[
  \C[x]_{\bw,0}=\C[x]^{\sigma}.
\]
We call $\bw$ \textbf{elliptic} if the $w_i$ are all nonzero and all of the same sign, \textbf{parabolic} if some $w_i$ vanishes, and \textbf{hyperbolic} if the $w_i$ are all nonzero and both signs occur. Every weight vector is of exactly one of these three types, and this trichotomy is the classification proposed in \cref{sec:9}.

Now let $\bq=(q_1,\dots,q_m)\in\Z^m$ be a second weight vector, let $y_1,\dots,y_m$ be the coordinates on the target, and write
\[
\taut(y_1,\dots,y_m)=\bigl(t^{q_1}y_1,\dots,t^{q_m}y_m\bigr)
\]
for the action associated to $\bq$. A polynomial map
\begin{equation}
\label{def:G}
\begin{split}
G\colon\A^n_{\bw}	& \longrightarrow  \A^m_{\bq} \\
(x_1, \ldots , x_n ) 	& \longmapsto \bigl(G_1(x), \ldots , G_m(x)\bigr)
\end{split}
\end{equation}
is called \textbf{graded} if each component $G_j$ is weight-homogeneous of weight $q_j$. We call $\bw$ the \textbf{source weights} and $\bq$ the \textbf{target weights}.

Equivalently, $G$ is graded if and only if it is \textbf{equivariant}, that is,
\[
G\circ\sigt=\taut\circ G \qquad (t\in\Gm),
\]
or, written out,
\[
G\bigl(t^{w_1}x_1,\dots,t^{w_n}x_n\bigr)=\bigl(t^{q_1}G_1(x),\dots,t^{q_m}G_m(x)\bigr)\qquad(t\in\Gm),
\]
which says that the diagram
\[
\begin{tikzcd}
\A^n_{\bw} \arrow[r, "\sigt"] \arrow[d, "G"'] & \A^n_{\bw} \arrow[d, "G"] \\
\A^m_{\bq} \arrow[r, "\taut"'] & \A^m_{\bq}
\end{tikzcd}
\]
commutes for every $t\in\Gm$. Write
\[
  G^{*}\colon\C[y_1,\dots,y_m]\longrightarrow\C[x_1,\dots,x_n], \qquad G^{*}f=f\circ G,
\]
for the comorphism of $G$, the $\C$-algebra homomorphism given by composition with $G$; thus $G^{*}y_j=G_j$. Since weights add under multiplication of polynomials, $G$ is graded if and only if $G^{*}$ preserves weights, that is,
\[
  G^{*}\bigl(\C[y]_{\bq,\ell}\bigr)\subseteq\C[x]_{\bw,\ell} \qquad\text{for every } \ell\in\Z ,
\]
so that $G^{*}$ is a homomorphism of graded rings of degree $0$.

Graded affine spaces and graded maps form a category: it is the full subcategory of $\Gm$-varieties whose objects are affine spaces with a diagonal linear action, and under $G\mapsto G^{*}$ it is anti-equivalent to the corresponding category of graded polynomial algebras and degree-zero homomorphisms. In particular $\id$ is graded, and $G\circ H$ is graded whenever $G$ and $H$ are, since $(G\circ H)^{*}=H^{*}\circ G^{*}$. We write $\Aut_{\bw}(\A^n)$ for the group of graded automorphisms of $\A^n_{\bw}$, that is the polynomial automorphisms of $\A^n$ commuting with $\sigma$; it is the group from which the liftable automorphisms of \cref{sec-8} are obtained.

 Two graded affine spaces $\A^n_{\bw}$ and $\A^n_{\bq}$ are carried to one another by a permutation of coordinates exactly when $\bq$ is a permutation of $\bw$.

Let $g=\gcd(w_1,\dots,w_n,q_1,\dots,q_m)$, and assume that $\bw$ and $\bq$ are not both zero, so that $g\geq1$. The subgroup of $\Gm$ acting trivially on both $\A^n_{\bw}$ and $\A^m_{\bq}$ is 
\[
\ker\sigma\cap\ker\tau=\muu{g}.
\]
 Both actions therefore factor through $\Gm/\muu{g}$, and the isomorphism $\Gm/\muu{g}\cong\Gm$ induced by $t\mapsto t^{g}$ carries them to the actions associated to $\bw/g$ and $\bq/g$. Since $g$ divides every $w_i$ and every $q_j$, the pieces of weight not divisible by $g$ are zero, while
\[
  \C[x]_{\bw,g\ell}=\C[x]_{\bw/g,\ell} \qquad\text{and}\qquad \C[y]_{\bq,g\ell}=\C[y]_{\bq/g,\ell}
\]
for every $\ell\in\Z$. The two pairs of weight vectors thus give the same decompositions, up to reindexing, and define the same graded maps.

For the remainder of this section we take $m=n$, so that the Jacobian matrix $\Jm{G}$ is square. By $\detJ {G}$ we denote its determinant.

\begin{lem}
\label{lem:grading}
Let $G$ be as in \cref{def:G}. The following hold:

i)  $\detJ{G}$ is weight-homogeneous of weight
\[
\wt ( \detJ {G}) = \sum q_j-\sum w_i.
\]
Moreover, if $\detJ{G}$ is a nonzero constant, then $\sum q_j=\sum w_i$.

ii)  If $q_j\neq0$ then $G_j(0)=0$; in particular $G(0)=0$ as soon as all $q_j$ are nonzero.

iii)  For all $i$ and $j$,
\[
\frac {\partial G_j} {\partial x_i} (0) =0
\]
 unless $q_j=w_i$.

\end{lem}

\begin{proof}
Differentiating the equivariance in $x$ gives
\[
\Jm{G}(\sigt p)\,\diag(t^{w_1},\dots,t^{w_n})=\diag(t^{q_1},\dots,t^{q_n})\,\Jm{G}(p),
\]
and hence
\[
\detJ{G}(\sigt p)=t^{\sum q_j-\sum w_i}\detJ{G}(p),
\]
which proves i), the second assertion because a nonzero constant has weight $0$.

Setting $x=0$ in the equivariance gives
\[
t^{q_j}G_j(0)=G_j(0)
\]
 for all $t$; since $q_j\neq0$, this forces $G_j(0)=0$, which is ii). Finally, the entry $\partial G_j/\partial x_i$ is weight-homogeneous of weight $q_j-w_i$, and $0$ is a fixed point of $\sigma$, so evaluating the corresponding scaling identity there gives
\[
\partial G_j/\partial x_i(0)=t^{\,q_j-w_i}\,\partial G_j/\partial x_i(0)
\]
for all $t$, which forces the entry to vanish unless $q_j=w_i$. This is iii), and no hypothesis on the weights is used.
\end{proof}

\begin{lem}
\label{lem:perm}
Let $G$ be as in \cref{def:G}, with $\detJ{G}(0)\neq0$. Then there is a permutation $\pi$ of $\{1,\dots,n\}$ with $q_j=w_{\pi(j)}$ for all $j$; if the $w_i$ are pairwise distinct, $\pi$ is unique and $\Jm{G}(0)$ is supported on the graph of $\pi$.
\end{lem}

\begin{proof}
Expand
\[
\detJ{G}(0)=\sum_{\pi}\sgn(\pi)\prod_j   \frac {\partial G_j} {\partial x_{\pi(j)}} (0).
\]
By part  iii) of  \cref{lem:grading} each factor vanishes unless $q_j=w_{\pi(j)}$, so a nonzero determinant requires at least one permutation along which all the equalities hold. If the $w_i$ are distinct, the condition $q_j=w_{\pi(j)}$ determines $\pi(j)$, and every other entry of $\Jm{G}(0)$ vanishes.
\end{proof}


A polynomial map $G\colon\C^n\to\C^n$ is called a \textbf{Keller map} if $\detJ{G}$ is a nonzero constant, and a \textbf{polynomial automorphism} if it is bijective with polynomial inverse. The Jacobian Conjecture asserts that every Keller map is a polynomial automorphism. We call $G$ a \textbf{graded Keller map} if in addition it is graded for some pair of weight vectors $\bw,\bq$.

\begin{cor}\label{cor:keller-weights}
Let $G$ be a graded Keller map with source weights $\bw$ and target weights $\bq$. Then $\bq$ is a permutation of $\bw$. Consequently:

\begin{enumerate}[label=\roman*)]
\item $\sum q_j=\sum w_i$;
\item a permutation of the coordinates is a graded isomorphism $\A^n_{\bw}\to\A^n_{\bq}$;
\item $\bw$ and $\bq$ are of the same type, elliptic, parabolic or hyperbolic;
\item $\gcd(q_1,\dots,q_n)=\gcd(w_1,\dots,w_n)$, so the normalisation $g=1$ says exactly that $\gcd(w_1,\dots,w_n)=1$.
\end{enumerate}
\end{cor}

\begin{proof}
Since $\detJ{G}$ is a nonzero constant we have $\detJ{G}(0)=\detJ{G}\neq0$, so \cref{lem:perm} provides a permutation $\pi$ with $q_j=w_{\pi(j)}$ for all $j$; that is, $(q_1,\dots,q_n)$ is a rearrangement of $(w_1,\dots,w_n)$. Now i) holds because equal multisets have equal sums, refining \cref{lem:grading}(i), and ii) is the criterion for a permutation of coordinates to be graded, recorded above. 

For iii) and iv), the type of a weight vector depends only on the signs of its entries and the $\gcd$ only on their values, neither being affected by the ordering and 
\[
  g=\gcd\bigl(\gcd(w_1,\dots,w_n),\,\gcd(q_1,\dots,q_n)\bigr)=\gcd(w_1,\dots,w_n).
\]
\end{proof}

\subsection{A counterexample to the Jacobian conjecture}

We now specialize to the map $F\colon\C^3\to\C^3$ of \cref{eq:map}. 
Take $n=3$, source weights $\bw=(1,-1,-2)$ on $\C[x,y,z]$ and target weights $\bq=(-2,-1,1)$ on $\C[a,b,c]$, so that
\[
  \wt(x^{i}y^{j}z^{\ell})=i-j-2\ell ,
\]
and $xy$ and $x^2z$ have weight $0$. Each component of $F$ is weight-homogeneous, with $\wt(a)=-2$, $\wt(b)=-1$, $\wt(c)=1$, so $F$ is graded, and by the above this is the same as
\begin{equation}\label{eq:equiv}
F\circ\sigt=\taut\circ F \qquad (t\in\Gm),
\end{equation}
where 
\[
\sigt(x,y,z)		=\bigl(tx,\,t^{-1}y,\,t^{-2}z\bigr) \qquad
\taut(a,b,c)		=\bigl(t^{-2}a,\,t^{-1}b,\,tc\bigr).
\]
Both weight vectors are hyperbolic. The grading is unique up to scaling: weight-homogeneity of the three monomials of $c$ alone forces $w_2=-w_1$ and $w_3=-2w_1$, so $\bw$ is a multiple of $(1,-1,-2)$ and is an invariant of the map.

The target weight vector $(-2,-1,1)$ is the reversal of the source weight vector $(1,-1,-2)$, so
\[
\sum q_j=\sum w_i=-2,
\]
 and the weights are distinct, so by \cref{lem:perm} the permutation $\pi$ is the reversal and, by \cref{lem:grading}(iii), the linear part of $F$ at the origin is \emph{anti-diagonal}:
 \[
\frac { \partial a} {\partial z}=1, \quad \frac {\partial b} {\partial y} =1, \quad \frac {\partial c} {\partial x}=2.
 \]
Since $\detJ{F}$ is constant it equals $\detJ{F}(0)$, and since the reversal permutation on three letters is odd,
\[
\detJ {F} \;=\; \detJ {F}(0) \;=\; -(1\cdot 1\cdot 2) \;=\; -2 .
\]
The constant $-2$ is the product of the three anti-diagonal leading coefficients.

\section{Two cases of the Jacobian Conjecture}
\label{sec-3}

In two cases a graded Keller map is a polynomial automorphism for a reason visible in the weights alone: elliptic weights in every dimension, where the orbits give properness, and dimension two for every signature, where the invariant ring is generated by one element and a degree count closes the argument. What remains is the hyperbolic case in dimension three and above, and \cref{eq:map} is an instance of it.

\subsection{The elliptic case in any dimension}\label{subsec:posrigid}

\begin{theorem}\label{thm:rigid}
Let $G\colon\A^n_{\bw}\to\A^n_{\bq}$ be a graded Keller map with elliptic weights $\bw$. Then $G$ is an isomorphism of graded affine spaces: it is bijective, its inverse is polynomial, and $G^{-1}\colon\A^n_{\bq}\to\A^n_{\bw}$ is graded. Forgetting the gradings, $G$ is a polynomial automorphism of $\A^n$.
\end{theorem}

\begin{proof}
By \cref{cor:keller-weights} the target weights $\bq$ are a permutation of $\bw$, hence also elliptic and of the same sign. If all the weights are negative, replacing $t$ by $t^{-1}$ in both actions replaces $(\bw,\bq)$ by $(-\bw,-\bq)$ and changes neither the map nor its being graded; so we may assume $w_i>0$ for all $i$, and then $q_j>0$ for all $j$.

Since $\detJ{G}$ is a nonzero constant, $G$ is \'etale. Steps 1--3 argue on complex points, so we write $x,y,p$ for points of $\A^n(\C)=\C^n$ and use the Euclidean topology throughout; there $G$ is a local biholomorphism.

\emph{Step 1: $G^{-1}(0)=\{0\}$.} The weights $q_j$ are nonzero, so $G(0)=0$ by \cref{lem:grading}(ii). Suppose $G(p)=0$ with $p\neq0$. The orbit map $t\mapsto\sigt(p)$ is non-constant: $\sigt(p)=p$ for all $t$ would force $t^{w_i}p_i=p_i$ for all $t$ and all $i$, hence $p=0$ as all $w_i>0$. But
\[
G(\sigt p)=\taut G(p)=\taut(0)=0,
\]
so the infinite connected set $\{\sigt(p):t\in\C^{\times}\}$ lies in $G^{-1}(0)$, which is discrete because $G$ is \'etale --- a contradiction.

\emph{Step 2: $G$ is proper.} Define
\[
\rho(x)=\max_i|x_i|^{1/w_i} \quad \text{ and } \quad \rho'(y)=\max_j|y_j|^{1/q_j},
\]
which make sense because all $w_i$ and all $q_j$ are positive. Both are continuous and vanish only at the origin. For $t\in\R_{>0}$ we have
\[
\rho(\sigt x)=t\rho(x) \quad \text{ and } \quad \rho'(\taut y)=t\rho'(y).
\]
The slice $S=\{\rho=1\}$ is closed and bounded, hence compact, and every $p\neq0$ is $\sigt(\bar p)$ with $t=\rho(p)$ and $\bar p=\sigma_{1/t}(p)\in S$. By Step~1, $0\notin G(S)$, so
\[
\delta:=\min_{y\in G(S)}\rho'(y)>0
\]
by compactness of $G(S)$. Let $K\subseteq\C^n$ be compact and suppose $p_k\in G^{-1}(K)$ with $|p_k|\to\infty$, so $t_k:=\rho(p_k)\to\infty$. Then
\[
\rho'\bigl(G(p_k)\bigr)=\rho'\bigl(\tau_{t_k}G(\bar p_k)\bigr)=t_k\,\rho'\bigl(G(\bar p_k)\bigr)\ \ge\ t_k\,\delta\ \longrightarrow\ \infty,
\]
contradicting $G(p_k)\in K$. Hence $G^{-1}(K)$ is bounded, and it is closed since $G$ is continuous, so it is compact and $G$ is proper.

\emph{Step 3: conclusion.} A proper local homeomorphism onto a connected, locally compact Hausdorff space is a covering map with finite fibers; hence $G$ is a finite covering of $\C^n$ by $\C^n$. The base is simply connected, so the covering is trivial; the total space is connected, so it has one sheet, and $G$ is bijective. An injective polynomial map $\C^n\to\C^n$ is a polynomial automorphism \cite{rudin}, so $G^{-1}$ is polynomial. Applying $G^{-1}$ on both sides of $G\circ\sigt=\taut\circ G$ gives
\[
G^{-1}\circ\taut=\sigt\circ G^{-1} \qquad (t\in\Gm),
\]
so $G^{-1}$ is graded with source weights $\bq$ and target weights $\bw$.
\end{proof}

\begin{rem}\label{rem:graded-nn}
\Cref{thm:rigid} has a reading outside algebraic geometry. In the graded architectures of \cite{gnn,gvs,gt}, the data carried by a node is a graded vector space and the admissible maps are those respecting the grading, so that the symmetry group of a layer is the group $\Aut_{\bw}(\A^n)$ of \cref{sec-2}. For a polynomial layer map of elliptic weight type, \cref{thm:rigid} says that a nonzero constant Jacobian determinant already forces bijectivity, with polynomial inverse, and the inverse is again graded with the same weights. Constancy of $\detJ{G}$ is the condition one imposes in any case when the log-determinant is to be tractable; in the positive-weight setting invertibility is not an extra hypothesis but a consequence of it. No such statement is available without a grading, where it is the Jacobian Conjecture itself. \Cref{thm:plane} extends it to every signature in dimension two, and the rest of this paper shows that both statements are sharp: the map \cref{eq:map} is graded for the mixed-sign weights $(1,-1,-2)$, has constant Jacobian determinant, and is three-to-one. What separates the two behaviours is the signature of the grading and nothing else.
\end{rem}

Elliptic weights are the setting of weighted projective spaces and of the Veronese morphisms between them; \cref{thm:rigid} says the Jacobian Conjecture is trivially true there for graded maps. A graded counterexample is therefore forced into parabolic or hyperbolic weights, and by \cref{thm:plane} into dimension at least three; the weights $(1,-1,-2)$ of \cref{eq:map} are the smallest for which we know an example to exist.

\begin{rem}\label{rem:vasyunin}
Hyperbolic weights are necessary but not sufficient. Let $G=X+H$ on $\C^{5}$, where
\[
  H=\bigl(0,\;X_1X_3,\;X_1X_4+\tfrac12X_2^{2},\;X_1X_5-X_2X_3,\;\tfrac12X_3^{2}\bigr)
\]
is the map of Vasyunin, reported in \cite{meistersolech}. Weight-homogeneity of the four nonlinear components forces $\bw=\bq=(1,-1,-2,-3,-4)$, so $G$ is graded with hyperbolic weights, and $\Jm{H}$ is nilpotent, so $\detJ{G}=1$ and $G$ is a Keller map. It is an automorphism, since $\deg G=2$ \cite{wang}, but it is not linearly triangularizable \cite{meistersolech}, and it is the extremal example for the strong nilpotence results of \cite{patecheng}. The examples of \cite[\S4]{patecheng} in every dimension are graded in the same way, with hyperbolic or parabolic weights, never elliptic. We return to this map in \cref{rem:vasyunin-descent}.
\end{rem}

\subsection{Dimension two, all signatures}\label{subsec:planethm}

In dimension two there is no graded counterexample to the Jacobian Conjecture, for any weight vector. We prove more: every Keller map of the plane admitting a nontrivial $\Gm$-equivariance is, after a change of coordinates on source and target, one of three normal forms, listed by the type of the weights.

\begin{theorem}\label{thm:plane}
Let $G\colon\A^2\to\A^2$ be a Keller map. Let $\theta$ and $\eta$ be algebraic actions of $\Gm$ on the source and on the target copy of $\A^2$, neither assumed diagonal in the given coordinates, with $\theta$ nontrivial, and suppose
\[
  G\circ\theta_t=\eta_t\circ G \qquad (t\in\Gm).
\]
Then $G$ is a polynomial automorphism. More precisely, there are polynomial automorphisms $\varphi,\psi$ of $\A^2$ such that $\varphi$ carries $\theta$ and $\psi$ carries $\eta$ to diagonal actions $\sigma$ and $\tau$, with source weights $\bw=(w_1,w_2)$, and $\psi\circ G\circ\varphi^{-1}$ is:
\begin{enumerate}
\item triangular if $\bw$ is elliptic, and linear if moreover $w_1=w_2$;
\item affine, if $\bw$ is parabolic;
\item linear, of the form $(x,y)\mapsto(a_0x,b_0y)$ or $(x,y)\mapsto(a_0y,b_0x)$, if $\bw$ is hyperbolic.
\end{enumerate}
\end{theorem}

\begin{proof}
The actions $\theta$ and $\eta$ are not assumed diagonal, so no weight vector is attached to them at the outset; the weight vectors of \cref{sec-2} appear only after linearization. Every algebraic $\Gm$-action on $\A^2$ is linearizable by a polynomial automorphism \cite{gutwirth}, and a linear action is diagonalizable, every rational representation of $\Gm$ being a sum of characters; so we may choose $\varphi,\psi$ for which
\[
  \sigt=\varphi\circ\theta_t\circ\varphi^{-1}
  \qquad\text{and}\qquad
  \taut=\psi\circ\eta_t\circ\psi^{-1}
\]
are diagonal. Put $G':=\psi\circ G\circ\varphi^{-1}$. Then
\[
  G'\circ\sigt
  =\psi\circ G\circ\theta_t\circ\varphi^{-1}
  =\psi\circ\eta_t\circ G\circ\varphi^{-1}
  =\taut\circ G' ,
\]
so $G'$ is graded in the sense of \cref{sec-2} for the weight vectors of $\sigma$ and $\tau$. Since a polynomial automorphism has constant nonzero Jacobian, $G'$ is again a Keller map, and $G$ is a polynomial automorphism if and only if $G'$ is. We may therefore assume from the outset that $G$ is graded, with source weights $\bw=(w_1,w_2)$ and target weights $\bq=(q_1,q_2)$; and we remain free to compose with a permutation of the source coordinates, absorbing it into $\varphi$, and with a permutation of the target coordinates, absorbing it into $\psi$.

By \cref{cor:keller-weights} the vector $\bq$ is a permutation of $\bw$, so composing with a permutation of the target coordinates we may assume $\bq=\bw$; in the hyperbolic case the alternative $\bq=(w_2,w_1)$ leads in the same way to the second normal form in (3), and we return to it at the end. Replacing $t$ by $t^{-1}$ in both actions replaces $\bw$ by $-\bw$ and changes neither $G$ nor its being graded, so we may fix the sign of $w_1$ as convenient. Since $\theta$ is nontrivial so is $\sigma$, and therefore $\bw\neq(0,0)$.

\emph{The elliptic case.} Here the $w_i$ are nonzero and of one sign; after the substitution above we may take them positive, and after a permutation of the source coordinates that $0<w_1\le w_2$. The weight-$\ell$ piece of $\C[x,y]$ is spanned by the monomials $x^iy^j$ with $iw_1+jw_2=\ell$, and since both weights are positive only finitely many $(i,j)$ occur.

Suppose first $w_1<w_2$. A monomial of weight $w_1$ has $j=0$, since $j\ge1$ would give $iw_1+jw_2\ge w_2>w_1$, and then $i=1$; so the weight-$w_1$ piece is $\C x$. A monomial of weight $w_2$ has $j\le1$; if $j=1$ then $i=0$, and if $j=0$ then $iw_1=w_2$, which is solvable only when $w_1$ divides $w_2$. Hence the weight-$w_2$ piece is $\C y$, together with $\C x^{w_2/w_1}$ in that case. Therefore
\[
  G=\bigl(a_0x,\;b_0y+b_1x^{w_2/w_1}\bigr), \qquad \detJ{G}=a_0b_0 ,
\]
the term $b_1x^{w_2/w_1}$ being absent unless $w_1\mid w_2$. The Keller condition says $a_0b_0\neq0$, so $G$ is triangular with nonzero diagonal, hence a polynomial automorphism. If instead $w_1=w_2$, then $iw_1+jw_2=w_1$ forces $i+j=1$, so both pieces equal $\C x+\C y$ and $G$ is linear, with invertible matrix because $\detJ{G}\neq0$.

\emph{The parabolic case.} Some $w_i$ vanishes, and not both, so after a permutation of the source coordinates $\bw=(w_1,0)$ with $w_1\neq0$, and after the substitution above $w_1>0$. A monomial $x^iy^j$ has weight $iw_1$, so it is invariant exactly when $i=0$ and of weight $w_1$ exactly when $i=1$: the weight-zero piece is $\C[y]$ and the weight-$w_1$ piece is $x\,\C[y]$. Hence
\[
  G=\bigl(xf(y),\,g(y)\bigr), \qquad f,g\in\C[y],
\]
and, the matrix being triangular,
\[
  \detJ{G}=f(y)\,g'(y).
\]
A product of two polynomials in $y$ is a nonzero constant only if both factors are nonzero constants, so $f=a_0$ and $g'=b_0$ with $a_0b_0\neq0$, whence $g=b_0y+b_1$. Thus $G=(a_0x,\,b_0y+b_1)$ is affine, and a polynomial automorphism.

\emph{The hyperbolic case.} The $w_i$ are nonzero of opposite signs, so after a permutation of the source coordinates and the substitution above $\bw=(p,-q)$ with $p,q>0$; and after the reduction by $g=\gcd$ of \cref{sec-2}, which does not change the graded maps, we may assume $\gcd(p,q)=1$. A monomial $x^iy^j$ has weight $ip-jq$, so it is invariant exactly when $ip=jq$, that is, since $\gcd(p,q)=1$, when $i=qk$ and $j=pk$ for some $k\ge0$. Hence
\[
  \C[x,y]^{\sigma}=\C[u], \qquad u=x^{q}y^{p},
\]
freely. The same computation with $ip-jq=p$ gives $i=1+qk$ and $j=pk$, and with $ip-jq=-q$ gives $i=qk$ and $j=1+pk$, so the weight-$p$ piece is $x\,\C[u]$ and the weight-$(-q)$ piece is $y\,\C[u]$. Therefore
\[
  G=\bigl(xf(u),\,yg(u)\bigr), \qquad f,g\in\C[u],
\]
with $f$ and $g$ nonzero, since a vanishing component would make $\detJ{G}=0$. From $u_x=qu/x$ and $u_y=pu/y$ one computes
\[
  \detJ{G}=\bigl(f+quf'\bigr)\bigl(g+pug'\bigr)-pq\,u^{2}f'g'
          =fg+u\bigl(pfg'+qgf'\bigr),
\]
the terms in $u^{2}f'g'$ cancelling. The invariant ring of the target action is likewise $\C[u]$, and $G^{*}u=G_1^{q}G_2^{p}=u\,f^{q}g^{p}$, so $G$ descends to the polynomial self-map
\[
  \Gbar(u)=u\,f(u)^{q}g(u)^{p}
\]
of $\A^1$, and the Leibniz rule gives
\[
  \Gbar'(u)=f^{\,q-1}g^{\,p-1}\Bigl[fg+u\bigl(pfg'+qgf'\bigr)\Bigr]
          =f^{\,q-1}g^{\,p-1}\,\detJ{G}.
\]
Now suppose $\detJ{G}=\kappa\in\C^{\times}$. Then $\Gbar$ has degree $d=1+q\deg f+p\deg g$, so $\Gbar'$ has degree $d-1$; on the other hand $\Gbar'=\kappa f^{\,q-1}g^{\,p-1}$ has degree
\[
  (q-1)\deg f+(p-1)\deg g=(d-1)-\deg f-\deg g .
\]
Comparing the two gives $\deg f+\deg g=0$, hence $\deg f=\deg g=0$. Thus $f=a_0$ and $g=b_0$ are constants with $a_0b_0=\kappa\neq0$, and $G=(a_0x,\,b_0y)$ is linear and a polynomial automorphism.

Finally, if $\bq=(w_2,w_1)$ rather than $\bq=\bw$, then composing $G$ with the transposition of the target coordinates, absorbed into $\psi$, returns to the case treated above. In the elliptic and parabolic cases the resulting normal forms are again triangular and affine respectively, and in the hyperbolic case one obtains $(x,y)\mapsto(a_0y,b_0x)$, the second form in (3). In every case $G$ is a polynomial automorphism.
\end{proof}
 
\section{Descent and the fiber cubic}
\label{sec-4}

\cref{sec-2} and \cref{sec-3} concerned graded Keller maps in general. We now turn to the single map \cref{eq:map} and study it for the next three sections. The organizing fact is that the grading reduces its fibres to the roots of one cubic in one variable, with coefficients in the invariants of the target action; everything we prove about \cref{eq:map} is read off that cubic.

The invariant rings of the two actions are polynomial:
\[
\C[x,y,z]^{\sigma}=\C[u,v],\quad u=xy,\; v=x^2z; \qquad \C[a,b,c]^{\tau}=\C[P,Q],\quad P=bc,\; Q=ac^2 .
\]
Generic orbits are free and closed, and by \cref{eq:equiv} the map $F$ carries orbits bijectively to orbits, so $F$ descends to $\Fbar\colon\A^2_{u,v}\to\A^2_{P,Q}$.
Set
\[
L=2-3u-v,\qquad M=v(1+u)^2+u^2(4+3u),
\]
so that $c=xL$ and, directly from \cref{eq:map},
\begin{equation}\label{eq:PQ}
P=L\,(u+3M),\qquad Q=L^2(1+u)\,M .
\end{equation}
Substituting $v=2-3u-L$ into $M$, the cubic terms cancel identically:
\begin{equation}\label{eq:M}
M=(2+u)-L(1+u)^2 .
\end{equation}

\begin{lem}\label{prop:identities}
Set $e=1+u$ and $s=Le$. Then:
\begin{enumerate}
\item the identities \cref{eq:PQ} become
\[
P=4s+2L-3s^2, \qquad Q=s^2+sL-s^3 ;
\]
\item eliminating $L$,
\begin{equation}\label{eq:cubic}
s^3-2s^2+P\,s-2Q=0,\qquad P=bc,\ Q=ac^2 ;
\end{equation}

\item of the roots of \cref{eq:cubic} with $L\neq0$ the preimage is recovered rationally:
\[
L=\tfrac12\!\left(P-4s+3s^2\right),\quad e=s/L,\quad u=e-1,\quad v=2-3u-L,
\]
\[
x=c/L,\qquad y=u/x,\qquad z=v/x^2 ;
\]

\item the elementary symmetric functions of a fiber are
\[
s_1+s_2+s_3=2,\qquad \textstyle\sum_{i<j}s_is_j=bc,\qquad s_1s_2s_3=2ac^2 .
\]
\end{enumerate}
\end{lem}

\begin{proof}
Using \cref{eq:M},
\[
u+3M=(e-1)+3(1+e-Le^2)=4e+2-3Le^2,
\]
so
\[
P=L(4e+2-3Le^2)=4s+2L-3s^2.
\]
Similarly
\[
Q=L^2e\,(1+e-Le^2)=sL+s^2-s^3.
\]
Solving the first identity for $L$ and substituting into the second gives \cref{eq:cubic}. The recovery formulas invert the substitutions step by step, using $c=xL$; and (4) reads off the coefficients of \cref{eq:cubic}.
\end{proof}

\begin{lem}\label{lem:basechange}
Let $E/F$ be a finite field extension and $T$ transcendental over $E$. Then $[E(T):F(T)]=[E:F]$.
\end{lem}

\begin{proof}
A basis $e_1,\dots,e_n$ of $E/F$ spans $E[T]$ over $F[T]$, hence $E(T)$ over $F(T)$. If
\[
\sum_i f_i(T)\,e_i=0
\]
with $f_i\in F[T]$ not all zero, then comparing coefficients of each power of $T$ gives $F$-linear relations among the $e_i$, forcing every coefficient of every $f_i$ to vanish. So the $e_i$ remain independent over $F(T)$.
\end{proof}

\begin{prop}\label{prop:degree3}
Write $\Omega=\C(x,y,z)$ and let $K=\C(a,b,c)\subseteq\Omega$ be the subfield generated by the components of $F$. Then:
\begin{enumerate}
\item $a,b,c$ are algebraically independent over $\C$, so $K$ is rational of transcendence degree $3$;
\item $[\Omega:K]=3$;
\item the cubic \cref{eq:cubic} is irreducible over $K$, hence also over its subfield $\C(P,Q)$, and is the minimal polynomial of $s$ over both.
\end{enumerate}
In particular the generic fiber of $F$ has exactly three points, independently of the announcement.
\end{prop}

\begin{proof}
(1) In characteristic zero, $a,b,c\in\C(x,y,z)$ are algebraically independent if and only if
\[
\det\bigl(\partial(a,b,c)/\partial(x,y,z)\bigr)\neq0
\]
where the determinant is $-2$.

(2) We first identify the invariant subfields. On $\{x\neq0\}$ the action $\sigma$ is free, and each orbit meets the slice $\{x=1\}$ exactly once, at
\[
\sigma_{1/x}(x,y,z)=(1,u,v),
\]
a $\sigma$-invariant rational function is therefore determined by its restriction to the slice, where $u,v$ are coordinates. Hence $\Omsig=\C(u,v)$.

The same argument on the target, with the slice $\{c=1\}$ on which $P,Q$ restrict to $b,a$, gives $\C(a,b,c)^{\tau}=\C(P,Q)$; and since $F^*$ intertwines the actions by \cref{eq:equiv}, the inclusion $K\subseteq\Omega$ restricts to
\[
\Ktau=\C(P,Q)\subseteq\C(u,v)=\Omsig.
\]
Next, $[\C(u,v):\C(P,Q)]=3$. Indeed
\[
\C(u,v)=\C(e,L)=\C(s,L)=\C(P,s):
\]
the first equality because $e=1+u$, $L=2-3u-v$ is a triangular change of variables, the second because $s=Le$ and $e=s/L$, the third because
\[
L=\tfrac12(P-4s+3s^2)
\]
while $P=4s+2L-3s^2$ (\cref{prop:identities}). Thus $P,s$ are algebraically independent, and by \cref{eq:cubic}, $Q=\tfrac12(s^3-2s^2+Ps)$ is a polynomial of degree $3$ in $s$ over $\C(P)$; therefore
\[
[\C(u,v):\C(P,Q)]=[\C(P)(s):\C(P)(Q)]=3,
\]
the degree of the rational map $s\mapsto Q$ over $\C(P)$.

Finally, $c=xL$ is transcendental over $\Omsig$, since
\[
\Omega=\C(x,u,v)=\Omsig(c)
\]
has transcendence degree $3=\trdeg\Omsig+1$; and
\[
K=\C(c,P,Q)=\Ktau(c)
\]
because $a=Q/c^2$ and $b=P/c$. \cref{lem:basechange} applied to $\C(u,v)/\C(P,Q)$ with $T=c$ gives
\[
[\Omega:K]=\bigl[\Omsig(c):\Ktau(c)\bigr]=\bigl[\Omsig:\Ktau\bigr]=3 .
\]

(3) The recovery formulas of \cref{prop:identities}(3) express $x,y,z$ rationally in $s$ and $a,b,c$, so $K(s)=\Omega$ and $[K(s):K]=3$. The minimal polynomial of $s$ over $K$ has degree $3$ and divides the monic cubic \cref{eq:cubic} in $K[s]$; being of the same degree, it equals it. Irreducibility over $K$ implies irreducibility over the subfield $\C(P,Q)$, since a factorization over the smaller field is one over the larger. The final statement holds because the generic fiber cardinality of a dominant map of irreducible complex varieties equals the degree of the function-field extension.
\end{proof}

\begin{theorem}\label{thm:main}
For the map \cref{eq:map}:
\begin{enumerate}
\item $\C(x,y,z)=\C(a,b,c)(s)$ with minimal polynomial \cref{eq:cubic}; in particular $\deg F=\deg\Fbar=3$.
\item The discriminant of \cref{eq:cubic},
\[
\Delta(P,Q)=-4P^3+4P^2+72PQ-64Q-108Q^2,
\]
is not a square in $\C(P,Q)$, so the geometric monodromy group of $F$ is $S_3$. Since $S_3$ is solvable, the inverse branches of $F$ are expressible by radicals in $a,b,c$.
\end{enumerate}
\end{theorem}

\begin{proof}
(1) is \cref{prop:degree3}. For (2),
\[
\Delta(P,0)=4P^2(1-P)
\]
is not a square, and an irreducible separable cubic with nonsquare discriminant (\cref{prop:degree3}(3)) has Galois group $S_3$.
\end{proof}

\section{Non-properness and the discriminant}
\label{sec:disc}

The cubic \cref{eq:cubic} makes the failure of properness visible: as a target point moves onto the discriminant, no two preimages ever meet. We now make this precise.

\begin{prop}\label{prop:disc}
\begin{enumerate}
\item $\Fbar$ contracts the line $\{L=0\}\subset\A^2_{u,v}$ to the origin $(P,Q)=(0,0)$.
\item A root $s_0$ of \cref{eq:cubic} is a multiple root if and only if $L(s_0)=\tfrac12(P-4s_0+3s_0^2)=0$. The discriminant locus $\{\Delta=0\}$ is the rational curve
\[
(P,Q)=\bigl(4s-3s^2,\ s^2-s^3\bigr),\qquad s\in\A^1 ,
\]
whose only singular point is the cusp $(P,Q)=(\tfrac43,\tfrac4{27})$, the image of $s=\tfrac23$.
\item Let $q=(a,b,c)$ with $c\neq0$. If $(P,Q)(q)\notin\{\Delta=0\}$ then $F^{-1}(q)$ has three points. If $(P,Q)(q)\in\{\Delta=0\}$ and $(P,Q)(q)\neq(\tfrac43,\tfrac4{27})$, then \cref{eq:cubic} has a double root with $L=0$, contributing no preimage, and a simple root with $L\neq0$, so $F^{-1}(q)$ has a single point. If $(P,Q)(q)=(\tfrac43,\tfrac4{27})$, then \cref{eq:cubic} has the triple root $s=\tfrac23$, at which $L=0$, and $F^{-1}(q)=\emptyset$. Along a path approaching a target on $\{\Delta=0\}$, every preimage whose root tends to a root with $L=0$ satisfies
\[
L\to0, \qquad x=c/L\to\infty ,
\]
and $e=s/L\to\infty$ as well when that root is nonzero: such preimages escape to infinity. No ramification occurs, consistent with \'etaleness of $F$; the ramification that a finite degree-three cover would require is converted entirely into non-properness.
\end{enumerate}
\end{prop}

\begin{proof}
(1) On $\{L=0\}$ we have $s=Le=0$, hence
\[
P=4s+2L-3s^2=0 \qquad \text{ and } \qquad Q=s^2+sL-s^3=0.
\]
(2) If $s_0$ is a multiple root then $3s_0^2-4s_0+P=0$, i.e.\ $P=4s_0-3s_0^2$, whence $L(s_0)=0$; conversely if $s_0$ is a root with $L(s_0)=0$ then $P=4s_0-3s_0^2$, so
\[
\tfrac{d}{ds}\bigl(s^3-2s^2+Ps-2Q\bigr)\big|_{s_0}=3s_0^2-4s_0+P=0 ,
\]
and $s_0$ is multiple; substituting into \cref{eq:cubic} gives $Q=s_0^2-s_0^3$. The parametrization follows, and its derivative $(4-6s,\,s(2-3s))$ vanishes only at $s=\tfrac23$, where $(P,Q)=(\tfrac43,\tfrac4{27})$.

(3) A preimage of $q$ requires $L\neq0$, since $c=xL\neq0$, and by \cref{prop:identities}(3) each root with $L\neq0$ gives exactly one preimage; by (2) the roots with $L=0$ are exactly the multiple ones. If $\Delta\neq0$ all three roots are simple, hence all have $L\neq0$. Suppose $\Delta=0$. The cubic has a triple root $s_0$ only if $3s_0=2$, and then $P=3s_0^2=\tfrac43$ and $2Q=s_0^3=\tfrac8{27}$; this is the cusp of (2), and there all three roots equal $\tfrac23$, where $L=0$, so the fiber is empty. Away from the cusp the cubic has one double root, with $L=0$, and one simple root, with $L\neq0$. The remaining assertions follow from the recovery formulas of \cref{prop:identities}(3), the second one because $e=s/L$ tends to infinity when $s$ tends to a nonzero limit and $L\to0$.
\end{proof}

In particular $F$ is not surjective, and what it misses is a single $\tau$-orbit,
\[
  \bigl\{\bigl(\tfrac4{27}t^{-2},\;\tfrac43t^{-1},\;t\bigr):t\in\C^{\times}\bigr\} .
\]
This is a second failure of bijectivity over $\C$, alongside the generic three-to-one behaviour, and it is invisible on the quotient: the cusp is an ordinary point of $\A^2_{P,Q}$, and $\Fbar$ does hit it, along $\{L=0\}$ only in the limit.

The hypothesis $c\neq0$ in \cref{prop:disc}(3) is essential: the whole plane $\{c=0\}$ maps to the single point $(P,Q)=(0,0)$, because $P=bc$ and $Q=ac^2$ vanish identically there, so the invariants fail to separate $\tau$-orbits on $\{c=0\}$ and the quotient picture is not a faithful guide. The fibers there are computed directly.

\begin{prop}\label{prop:cfibres}
Let $(\alpha,\beta,0)\in\C^3$. Then $F^{-1}(\alpha,\beta,0)$ consists of the point $(0,\beta,\alpha-4\beta^2)$, lying on $\{x=0\}$, together with the preimages at which $L$ vanishes, determined as follows. If $\beta\neq0$, they are given by the roots of
\begin{equation}\label{eq:cquad}
(16\alpha-\beta^2)\,u^2+(48\alpha-3\beta^2)\,u+(36\alpha-2\beta^2)=0,\qquad x=\frac{4u+6}{\beta},
\end{equation}
and if $\beta=0$, by $u=-\tfrac32$ together with $x^2=-\tfrac1{4\alpha}$. Consequently
\[
\#F^{-1}(\alpha,\beta,0)=\begin{cases}1,&\beta^2=16\alpha,\\ 3,&\text{otherwise.}\end{cases}
\]
\end{prop}

\begin{proof}
Since $c=xL$, a preimage has $x=0$ or $L=0$. On $\{x=0\}$ the map restricts to
\[
(y,z)\mapsto(z+4y^2,\,y,\,0),
\]
a bijection onto $\{c=0\}$, giving the point $(0,\beta,\alpha-4\beta^2)$. At a preimage with $L=0$ we have $x\neq0$, and $v=2-3u$ with $M=2+u$ by \cref{eq:M}, so the invariants $bx$ and $ax^2$, of weight zero, are
\[
bx=u+3M=4u+6,\qquad ax^2=(1+u)M=(1+u)(2+u).
\]
Suppose $\beta\neq0$. Then $x=(4u+6)/\beta$, and $\alpha x^2=(1+u)(2+u)$ gives \cref{eq:cquad}; its discriminant is
\[
(48\alpha-3\beta^2)^2-4(16\alpha-\beta^2)(36\alpha-2\beta^2)=\beta^2\bigl(\beta^2-16\alpha\bigr).
\]
If $\beta^2\neq16\alpha$ there are two distinct roots, and neither is $u=-\tfrac32$: substituting $u=-\tfrac32$ into the left side of \cref{eq:cquad} gives
\[
\tfrac94(16\alpha-\beta^2)-\tfrac32(48\alpha-3\beta^2)+(36\alpha-2\beta^2)=\tfrac14\beta^2\neq0 ,
\]
the terms in $\alpha$ cancelling. Hence each root determines $x\neq0$, and with it $y=u/x$ and $z=v/x^2$, so it determines a single point; distinct roots give distinct points, since $u=xy$ is a function of the point. If instead $\beta^2=16\alpha$, then $\alpha=\beta^2/16\neq0$, the coefficients of $u^2$ and of $u$ in \cref{eq:cquad} both vanish, and the equation degenerates to $4\alpha=0$, which has no solution.

Suppose $\beta=0$. Then $bx=0$ forces $u=-\tfrac32$, whence
\[
ax^2=\bigl(-\tfrac12\bigr)\bigl(\tfrac12\bigr)=-\tfrac14 \qquad \text{ and } \qquad x^2=-\tfrac1{4\alpha},
\]
giving two points when $\alpha\neq0$, distinguished by the sign of $x$, and none when $\alpha=0$. In every case the points with $L=0$ have $x\neq0$, so they are distinct from the point on $\{x=0\}$, and the stated cardinalities follow, the exceptional case $\beta^2=16\alpha$ covering both $\beta\neq0$ with $\alpha=\beta^2/16$ and $\beta=\alpha=0$.
\end{proof}

The announced collision sits over $(-\tfrac14,0,0)$, whose fiber consists of the three rational points
\[
(0,0,-\tfrac14),\qquad (1,-\tfrac32,\tfrac{13}2),\qquad (-1,\tfrac32,\tfrac{13}2).
\]
The third is not independent data. Since $b=c=0$ there, the target point is fixed by $\tau_{-1}$, so its fiber is stable under
\[
\sigma_{-1}(x,y,z)=(-x,-y,z),
\]
and the last two points form a free $\muu{2}$-orbit while the first is $\muu{2}$-fixed: the grading predicts the third point from the second. Both points of the free orbit have $(u,v)=(-\tfrac32,\tfrac{13}2)$ and lie on $\{L=0\}$, while the first has $(u,v)=(0,0)$, $L=2$ and $s=2$, the simple root of $s^2(s-2)=0$. By \cref{prop:cfibres} the fiber is not degenerate in cardinality; the drops described in \cref{prop:disc}(3) occur over targets with $c\neq0$, and within $\{c=0\}$ the cardinality falls only along the curve $\beta^2=16\alpha$.

\begin{rem}[A Kummer condition on the stacky line]\label{rem:kummer}
The action $\tau$ has stabilizer $\muu{2}$ exactly along the punctured line $\{b=c=0,\ a\neq0\}$, and $\sigma$ has stabilizer $\muu{2}$ along $\{x=y=0,\ z\neq0\}$. By \cref{prop:cfibres}, for $\alpha\in\Q^{\times}$ the two preimages of $(\alpha,0,0)$ lying on $\{L=0\}$ have $x^2=-\tfrac1{4\alpha}$, so
\[
\#F^{-1}(\alpha,0,0)(\Q)=\begin{cases}3,&-\alpha\in(\Q^{\times})^{2},\\ 1,&\text{otherwise,}\end{cases}
\]
the point $(0,0,\alpha)$ occurring in either case. Liftability of a rational point along the $\muu{2}$-fixed line is thus governed by a square class, an obstruction of exactly the Kummer type that governs the positive-weight setting~\cite{sparsity}. The announced collision is the member $\alpha=-\tfrac14$ of this family, and $\alpha=-1$, with preimages $(\pm\tfrac12,\mp3,26)$, is the next.
\end{rem}

\begin{rem}
The three roots of \cref{eq:cubic} always sum to $2=|\detJ {F}|$, the same constant produced by the anti-diagonal linear part in \cref{lem:grading}; equivalently, the coefficient of $s^2$ is $-2$. This is a statement about the roots and not about the fiber: on the discriminant the double root $s_0$ has $L(s_0)=0$ and is not a preimage, so the surviving point carries $s=2-2s_0$, which equals $2$ only when $s_0=0$. The same constant appears once more as the value $\kappa=2$ of \cref{thm:master}. Whether the $s^2$-coefficient of the fiber cubic equals $\kappa$ for every Keller map graded by this weight type, or whether these coincidences are accidents of \cref{eq:map}, we do not know.
\end{rem}

\section{Arithmetic: the image on rational points is thin}\label{sec:6}

Everything above is defined over $\Q$. If $(x,y,z)\in\Q^3$, then $s=Le\in\Q$; hence a point $(a,b,c)\in\Q^3$ with $c\neq0$ has a rational preimage if and only if the cubic \cref{eq:cubic} has a rational root with $L\neq0$.

\begin{cor}\label{cor:thin}
$F(\Q^3)$ is a thin set in the sense of Serre \cite{serre}. By Hilbert irreducibility, outside a thin set of rational points $(a,b,c)\in\Q^3$ the fiber $F^{-1}(a,b,c)$ contains no rational point.
\end{cor}

\begin{proof}
The source $\C^3$ is irreducible and $F$ is dominant of degree $3$ by \cref{prop:degree3}, so $F(\Q^3)$ is the image of the rational points of an irreducible variety under a generically finite map of degree greater than one; such images are thin by definition \cite[\S3.1]{serre}. The second statement is Hilbert's irreducibility theorem applied to the cubic \cref{eq:cubic}, which is irreducible over $\C(a,b,c)$ by \cref{prop:degree3}(3).
\end{proof}

Over $\C$ almost every point of the target has three preimages; over $\Q$ almost none has any. The counterexample therefore fails bijectivity on $\Q$-points in both directions at once: rational collisions, and a thin image. It is instructive to compare with the lifting problem for the Veronese morphism $\varphi\colon\WP^n_{\bq}\to\bP^n$ on weighted projective spaces~\cite{sparsity}.
There the weights are positive, $\varphi$ is finite and ramified, and liftability of a rational point is an abelian condition: a torsor under roots of unity, classified by Kummer classes and detected by congruences at every prime. Here the weights are hyperbolic, $F$ is \'etale and non-proper, and liftability is the splitting of a nonabelian $S_3$-resolvent, with no torsor structure and no congruence criterion --- except along the stabilizer locus, where \cref{rem:kummer} shows the abelian condition returns.

\section{Descent of the Keller condition}
\label{sec:7}

The identity $\detJ {F}=-2$ is equivalent to a two-dimensional statement about $\Fbar$ and the contracted line. The following holds for every equivariant map of this weight type, not only for \cref{eq:map}.

\begin{theorem}\label{thm:descent}
Let $G\colon\C^3\to\C^3$ be a dominant polynomial map equivariant for $\sigma,\tau$ as in \cref{eq:equiv}, with weights $(1,-1,-2)$ and degrees $(-2,-1,1)$. Write the third component as $G_3=x\,\Lambda(u,v)$ with $\Lambda$ invariant, and let $(P,Q)=(G_2G_3,\,G_1G_3^2)$ be the descended map on invariants. Then, up to a universal sign,
\[
\detJ{G} \;=\; \pm\,\frac{\Jacuv(P,Q)}{\Lambda^2}\, .
\]
In particular $G$ is a Keller map if and only if
\[
\Jacuv(P,Q)\;=\;\kappa\,\Lambda^2 \qquad\text{for some }\kappa\in\C^{\times},
\]
that is, if and only if the Jacobian of the quotient map vanishes exactly on the contracted locus $\{\Lambda=0\}$, to order two, with constant leading coefficient. For the map \cref{eq:map} we have $\Lambda=L$ and
\(
\Jacuv(P,Q)=2L^2 .
\)
\end{theorem}

\begin{proof}
The component $G_3$ has weight $1$, and every weight-one rational function on $\{x\neq0\}$ is $x$ times an invariant, so $G_3=x\Lambda$ with $\Lambda\in\C[u,v]$; likewise $P,Q$ are invariant. On $\{x\neq0\}$ use coordinates $(x,u,v)$, and on $\{G_3\neq0\}$ target coordinates $(c,P,Q)$, related to $(a,b,c)$ by $a=Q/c^{2}$, $b=P/c$. In these coordinates $G$ is
\[
(x,u,v)\mapsto\bigl(x\Lambda(u,v),\,P(u,v),\,Q(u,v)\bigr),
\]
with Jacobian $\Lambda\cdot\Jacuv(P,Q)$. The source change $(x,y,z)\to(x,u,v)$ has Jacobian $\pm x^{3}$ and the target change $(c,P,Q)\to(a,b,c)$ has Jacobian $\pm c^{-3}=\pm(x\Lambda)^{-3}$, whence
\[
\detJ{G}=\pm\,x^{3}\,(x\Lambda)^{-3}\,\Lambda\,\Jacuv(P,Q)=\pm\,\Lambda^{-2}\Jacuv(P,Q)
\]
on a dense open set, hence everywhere. The displayed value for \cref{eq:map} is a direct computation through the coordinates of \cref{prop:identities}: one finds
\[
\frac{\partial(P,Q)}{\partial(s,L)}=-2L, \qquad \frac{\partial(s,L)}{\partial(e,L)}=L, \qquad \frac{\partial(e,L)}{\partial(u,v)}=-1.
\]
\end{proof}

\begin{rem}\label{rem:recipe}
\cref{thm:descent} explains the cancellation \cref{eq:M}: it is the condition $\Jac=\kappa\Lambda^{2}$ made explicit. It also inverts into a recipe: to build an equivariant Keller map of this weight type, solve the two-dimensional problem $\Jacuv(P,Q)=\kappa\Lambda^{2}$ --- a Jacobian problem with prescribed divisorial vanishing along the contracted line --- and lift. More generally, it reduces the Jacobian problem for hyperbolically graded maps of $\C^{n}$ to a degenerate Jacobian problem on the $(n-1)$-dimensional quotient. We do not know whether the exponent $2$ admits a uniform description in terms of the weights; this is the first question the reduction raises.
\end{rem}


\section{The parameter variety of graded Keller maps}
\label{sec-8}

\Cref{thm:descent} reduces the construction of a graded Keller map to a problem on the invariant quotient. We make that reduction explicit for an arbitrary hyperbolic signature and observe that the resulting condition is \emph{trilinear} in its unknowns. The graded Keller maps of bounded degree are therefore the points of an explicit multiprojective variety defined over $\Z$, and the question of whether \cref{eq:map} is isolated becomes a question about the rational points of that variety. This is the form in which the two settings of \cref{sec:9} are most closely comparable: the arithmetic of the hyperbolic example is governed by a parameter space which is itself of the positive-weight type studied in \cite{sparsity}.

\subsection{Descent for an arbitrary signature}
Fix integers $r, s>0$. We have the action of $\Gm$ on $\C^3$ via \(\Gm \times \C^3 \to \C^3\) given by
\[
(t, (x_1, x_2, x_3) ) \mapsto \bigl(tx_1,\;t^{-r}x_2,\;t^{-s}x_3\bigr).
\]
For any $t\in\Gm$ denote by \( \sigt : \C^3 \to \C^3 \) the map given by
\[
  \sigt (x_1,x_2,x_3) = \bigl(tx_1,\;t^{-r}x_2,\;t^{-s}x_3\bigr).
\]
On the image $\C^3$ we have the action of $\Gm$ via \(\Gm \times \C^3 \to \C^3\) given by
\[
(t, (x_1, x_2, x_3) ) \mapsto \bigl(t^{-s}x_1,\;t^{-r}x_2,\;tx_3\bigr),
\]
that is the reversal degree type $(-s,-r,1)$ of \cref{lem:perm}. For any $t\in\Gm$ denote by \( \taut : \C^3 \to \C^3 \) the map given by
\[
  \taut(x_1,x_2,x_3) = \bigl(t^{-s}x_1,\;t^{-r}x_2,\;tx_3\bigr).
\]

By \cref{cor:keller-weights} the target weights of a graded Keller map are a permutation of the source weights, and by \cref{lem:perm} the permutation is determined once the $w_i$ are distinct. We fix the reversal here because it is the type of \cref{eq:map}, but nothing below depends on that choice. What the descent uses is only that exactly one coordinate on each side carries the positive weight $1$, and the computation is unchanged if that coordinate is given a different index. The identity type, in which the positive weight stays in the first position on both sides, occurs in \cref{rem:vasyunin} and is the setting of \cref{rem:exponent}.

 \begin{lem}
The elements $\sigt, \taut$ are conjugated in $\Aut(\C^3)$ by  
\[
  w=\begin{pmatrix} 0&0&1\\ 0&1&0\\ 1&0&0 \end{pmatrix} \in   \GL_3(\C).  
\]
In other words, 
\(
  \taut = w\circ\sigt\circ w^{-1}
  \)   for every \( t\in\Gm\).
\end{lem}

\begin{proof}
In the coordinates $x_1,x_2,x_3$ both actions are diagonal,
\[
  \sigt=\diag\bigl(t,\;t^{-r},\;t^{-s}\bigr), \qquad \taut=\diag\bigl(t^{-s},\;t^{-r},\;t\bigr).
\]
Since $w^2=\id$ we have $w^{-1}=w$, and for any $d_1,d_2,d_3$ one checks directly that
\[
  w\,\diag(d_1,d_2,d_3)\,w^{-1}=\diag(d_3,d_2,d_1).
\]
Applying this with $(d_1,d_2,d_3)=(t,t^{-r},t^{-s})$ gives $w\,\sigt\,w^{-1}=\taut$.
\end{proof}


\subsubsection{Action on the polynomial ring $\C[x,y,z]$}

Grade $\C[x,y,z]$ by the weight defined in \cref{sec-2}, extended to general $r,s$: the monomial $x^iy^jz^{\ell}$ has weight $i-rj-s\ell$. The action $\sigma$ extends to $\C[x,y,z]$ by $\sigt^{*}f=f\circ\sigt$, which is an action since $\Gm$ is commutative, and then $\sigt^{*}$ multiplies a monomial by $t^{\,i-rj-s\ell}$, that is by $t$ to its weight. We write $\C[x,y,z]^{\sigma}$ for the ring of polynomials invariant under $\sigt^{*}$ for every $t$, that is the polynomials of weight $0$. A monomial $x^i y^j z^{\ell}$ is invariant if and only if $i=r j+s \ell$, so
\[
  \C[x,y,z]^{\sigma}=\C[u,v],\qquad u=x^{r}y,\quad v=x^{s}z,
\]
freely, for every $r,s$.

In the same way $\tau$ extends to $\C[a,b,c]$ by $\taut^{*}f=f\circ\taut$, the monomial $a^ib^jc^{\ell}$ has weight $\ell-si-rj$, and we write $\C[a,b,c]^{\tau}$ for the ring of polynomials invariant under $\taut^{*}$ for every $t$. A monomial $a^i b^j c^{\ell}$ is invariant if and only if $\ell=s i+r j$, so
\[
  \C[a,b,c]^{\tau}=\C[P,Q],\qquad P=b\,c^{\,r},\quad Q=a\,c^{\,s},
\]
freely. For $(r,s)=(1,2)$ these are the invariants of \cref{sec:7}.

A polynomial map $G=(G_1,G_2,G_3)$ is equivariant for $(\sigma,\tau)$ exactly when each component is weight-homogeneous of the weight of the corresponding image coordinate, that is $G_1$ of weight $-s$, $G_2$ of weight $-r$, and $G_3$ of weight $1$: this is the content of the equivariance $G\circ\sigt=\taut\circ G$, read one coordinate at a time.

We describe the three graded pieces. A monomial of weight $1$ has $x$-exponent $i=1+rj+s\ell\ge1$, so it is $x$ times an invariant, and the weight-$1$ piece is $x\,\C[u,v]$. For the other two, write the monomials of $\C[u,v]$ as $u^iv^j=x^{ri+sj}y^iz^j$. Then $x^{-r}u^iv^j$ is a polynomial exactly when $ri+sj\ge r$, and $x^{-s}u^iv^j$ exactly when $ri+sj\ge s$, so the weight-$(-r)$ piece is $x^{-r}\VB$ and the weight-$(-s)$ piece is $x^{-s}\VA$, where
\[
  \VB:=\Span\bigl\{u^iv^j:\ ri+sj\ge r\bigr\},
  \qquad
  \VA:=\Span\bigl\{u^iv^j:\ ri+sj\ge s\bigr\}.
\]
Both are monomial ideals of $\C[u,v]$; for $(r,s)=(1,2)$ they are $(u,v)$ and $(u^2,v)$, the conditions appearing in \cref{sec:7}. Collecting the three pieces, an equivariant polynomial map $G\colon\C^3\to\C^3$ of this weight type is therefore exactly a triple
\begin{equation}\label{eq:triple}
  G=\bigl(x^{-s}A,\;x^{-r}B,\;x\Lambda\bigr),
  \qquad A\in\VA,\quad B\in\VB,\quad \Lambda\in\C[u,v],
\end{equation}
and the descended map on invariants is $\Gbar=(P,Q)=\bigl(B\Lambda^{r},\,A\Lambda^{s}\bigr)$.


\begin{prop}\label{prop:gendescent}
For $G$ as in \cref{eq:triple}, up to a universal sign,
\[
  \detJ{G}=\pm\,\frac{\Jacuv(P,Q)}{\Lambda^{\,r+s-1}} .
\]
\end{prop}

\begin{proof}
On $\{x\neq0\}$ use coordinates $(x,u,v)$; the change of coordinates $(x,y,z)\mapsto(x,u,v)$ has Jacobian $\pm x^{r+s}$. On $\{c\neq0\}$ use $(c,P,Q)$; the change $(a,b,c)\mapsto(c,P,Q)$ has Jacobian $\pm c^{\,r+s}$. In these coordinates $G$ is 
\[
(x,u,v)\mapsto\bigl(x\Lambda(u,v),P(u,v),Q(u,v)\bigr),
\]
 with Jacobian $\Lambda\Jacuv(P,Q)$. The chain rule and $c=x\Lambda$ give
\[
  \detJ{G}=\pm\,\frac{x^{r+s}\,\Lambda\Jacuv(P,Q)}
                    {(x\Lambda)^{r+s}}
         =\pm\,\frac{\Jacuv(P,Q)}{\Lambda^{\,r+s-1}}
\]
on a dense open set, hence everywhere.
\end{proof}

\begin{theorem}\label{thm:master}
Let $A,B,\Lambda\in\C[u,v]$ and write $\jb{f}{g}=f_ug_v-f_vg_u$. Then
\[
  \Jacuv\bigl(B\Lambda^{r},A\Lambda^{s}\bigr)
  =\Lambda^{\,r+s-1}\Bigl[\Lambda\jb{B}{A}+sA\jb{B}{\Lambda}+rB\jb{\Lambda}{A}\Bigr].
\]
Consequently the map $G$ of \cref{eq:triple} is a Keller map if and only if
\begin{equation}\label{eq:master}
  \Lambda\jb{B}{A}\;+\;sA\jb{B}{\Lambda}\;+\;rB\jb{\Lambda}{A}\;=\;\kappa
  \qquad\text{for some }\kappa\in\C^{\times}.
\end{equation}
\end{theorem}

\begin{proof}
By the Leibniz rule in each slot, 
\[
\jb{\Lambda^{r}}{A\Lambda^{s}}=\Lambda^{s}\jb{\Lambda^{r}}{A}=r\Lambda^{r+s-1}\jb{\Lambda}{A}
\]
 and 
 \[
 \jb{B}{A\Lambda^{s}}=sA\Lambda^{s-1}\jb{B}{\Lambda}+\Lambda^{s}\jb{B}{A},
 \]
  whence
\[
  \jb{B\Lambda^{r}}{A\Lambda^{s}}
  =B\jb{\Lambda^{r}}{A\Lambda^{s}}+\Lambda^{r}\jb{B}{A\Lambda^{s}}
  =\Lambda^{r+s-1}\bigl[rB\jb{\Lambda}{A}+sA\jb{B}{\Lambda}+\Lambda\jb{B}{A}\bigr].
\]
The second statement follows from \cref{prop:gendescent}.
\end{proof}

For $(r,s)=(1,2)$, \cref{eq:master} reads 
\[
\Lambda\jb{B}{A}+2A\jb{B}{\Lambda}+B\jb{\Lambda}{A}=\kappa,
\]
 and for \cref{eq:map} one computes $\kappa=2$, recovering $\Jacuv(P,Q)=2L^{2}$ of \cref{thm:descent}.

\begin{rem}\label{rem:exponent}
\Cref{prop:gendescent} answers the question raised in \cref{rem:recipe}: the exponent is determined by the signature. Let $\wt(x,y_2,\dots,y_n)=(1,-r_2,\dots,-r_n)$ with $r_i>0$, let $u_i=x^{r_i}y_i$ be the invariants, and let the target weights be the identity permutation of the source weights, so that the coordinate of weight $1$ is again the first. Writing $G_1=x\Lambda$ and $P_i=G_iG_1^{r_i}=B_i\Lambda^{r_i}$ for $i\ge2$, the same computation gives
\[
  \detJ{G}=\pm\,\frac{\Jac_{u}(P_2,\dots,P_n)}{\Lambda^{\,m-1}},
  \qquad m=\sum_{i\ge2}r_i .
\]
The exponent is $m-1$, which is $2$ for the weights $\wt=(1,-1,-2)$; for $n=3$ one has $m=r+s$, in agreement with \cref{prop:gendescent}. That the positive weight equals $1$ is what makes the quotient again an affine space: already for weights $(2,-1,-1)$ the invariant ring is $\C[xy^{2},xyz,xz^{2}]$, a quadric cone, and the reduction takes place on a singular quotient.
\end{rem}


\subsection{The base point}

The conditions $A\in\VA$, $B\in\VB$ are not merely technical; they pin down a marked point.

\begin{lem}\label{lem:basept}
Let $G$ be as in \cref{eq:triple} and suppose \cref{eq:master} holds. Let $O=(0,0)$ be the origin of $\A^2_{u,v}$. Then
\[
  A(O)=B(O)=0,\qquad \Lambda(O)\neq0,\qquad dA\wedge dB\big|_{O}\neq0,
\]
and $\Gbar(O)=(0,0)$.
\end{lem}

\begin{proof}
Every monomial $u^iv^j$ spanning $\VA$ satisfies $ri+sj\ge s>0$, and every one spanning $\VB$ satisfies $ri+sj\ge r>0$, so neither space contains a constant and both lie in the maximal ideal $(u,v)$. 
Hence $A(O)=B(O)=0$, and $\Gbar=(B\Lambda^{r},A\Lambda^{s})$ gives $\Gbar(O)=(0,0)$. Evaluating \cref{eq:master} at $O$, the terms $sA\jb{B}{\Lambda}$ and $rB\jb{\Lambda}{A}$ vanish because $A$ and $B$ do, and there remains
\[
  \Lambda(O)\,\jb{B}{A}(O)=\kappa\neq0 .
\]
Therefore $\Lambda(O)\neq0$ and $\jb{B}{A}(O)\neq0$; since $dA\wedge dB=-\jb{B}{A}\,du\wedge dv$, the last assertion follows.
\end{proof}

Thus $(A,B)$ is a pair of polynomials with a common zero at which it is \'etale, and $\Lambda$ does not vanish there. Conversely:

\begin{rem}\label{rem:normalize}
Let $r=1$ and let $A,B,\Lambda$ solve \cref{eq:master} with $A$ and $B$ having a common zero $p_0\in\A^2$. Then some polynomial automorphism of $\A^2$ carries the data into the normalised form of \cref{lem:basept}. Indeed \cref{eq:master} is an identity of polynomials, so it may be evaluated at $p_0$; since $A(p_0)=B(p_0)=0$ the last two terms vanish and
\[
  \Lambda(p_0)\,\jb{B}{A}(p_0)=\kappa\neq0 ,
\]
so $\jb{B}{A}(p_0)\neq0$ and the curves $\{A=0\}$ and $\{B=0\}$ are smooth at $p_0$ and meet transversally there. Translate $p_0$ to $O$. Since $dA|_{O}\neq0$, a linear change of coordinates makes $A_v(O)\neq0$, and it does not disturb $B(O)=0$. By the implicit function theorem $\{A=0\}$ is then the graph of a power series $v=f(u)$ with $f(0)=0$; replace $v$ by $v-\pi(u)$, where $\pi$ is the Taylor polynomial of $f$ of order $s$ at $O$. This is a polynomial automorphism of $\A^2$, and in the new coordinates $A(u,0)$ vanishes to order at least $s+1$, so $A\in(u^{s},v)=\VA$, while $B(O)=0$ gives $B\in(u,v)=\VB$. Finally, for any polynomial automorphism $\varphi$ of $\A^2$ one has
\[
  \jb{f\circ\varphi}{g\circ\varphi}=\bigl(\jb{f}{g}\circ\varphi\bigr)\,\det J\varphi ,
\]
so each of the three terms of \cref{eq:master} is multiplied by the same constant $\det J\varphi$, and the equation is preserved with $\kappa$ replaced by $\kappa\det J\varphi\in\C^{\times}$.
\end{rem}

\begin{theorem}\label{thm:autononinj}
Let $G$ be as in \cref{eq:triple}, satisfying \cref{eq:master}, with $\Lambda\notin\C$. Then $G$ is a counterexample to the Jacobian Conjecture, of degree $\deg G=\deg\Gbar\ge2$.
\end{theorem}

\begin{proof}
$G$ is Keller by \cref{thm:master}, hence dominant, and $\deg G=\deg\Gbar$ by the argument of \cref{prop:degree3}, which uses only $\Omega=\Omsig(c)$, $K=\Ktau(c)$ and \cref{lem:basechange}, and so applies verbatim to any $(r,s)$. Suppose $\deg G=1$. Then $G$ is birational and \'etale, hence an open immersion by Zariski's main theorem, and its image $U$ is isomorphic to $\C^{3}$, in particular affine; so $\C^3\setminus U$ is either empty or a hypersurface $V(h)$ with $h$ nonconstant, the complement of an affine open subset of a smooth variety being pure of codimension one. In the latter case $h$ is a nonconstant unit of $\Ost(U)$, contradicting $U\cong\C^{3}$. Hence $U=\C^3$ and $G$ is an automorphism.

Its inverse is again equivariant, by the argument in the last step of \cref{thm:rigid}, so $(G^{-1})^{*}$ preserves weights and therefore carries $\C[u,v]=\C[x,y,z]^{\sigma}$ into $\C[a,b,c]^{\tau}=\C[P,Q]$; this restriction is inverse to $\Gbar^{*}$, so $\Gbar$ is an automorphism of $\A^2$ and $\Jacuv(P,Q)$ is a nonzero constant. By \cref{thm:master} that constant equals $\kappa\Lambda^{\,r+s-1}$, and $r+s-1\ge1$, so $\Lambda$ is constant, against the hypothesis. Thus $\deg G\ge2$ and $G$ is not injective.
\end{proof}

In the graded setting there is therefore no separate injectivity condition to impose: non-injectivity is a consequence of \cref{eq:master} together with $\Lambda\notin\C$. The construction of graded counterexamples is the solution of one scalar equation.

\begin{rem}\label{rem:vasyunin-descent}
The map of \cref{rem:vasyunin} illustrates the role of $\Lambda$. There $\bq=\bw$, so the weight-one component is $G_1$ and the notation is that of \cref{rem:exponent}, with $(r_2,r_3,r_4,r_5)=(1,2,3,4)$ and $m=\sum_{i\ge2}r_i=10$. The invariants are $u_i=X_1^{r_i}X_i$, and $G_1=X_1$ gives $\Lambda=1$. Writing $P_i=G_iG_1^{r_i}=B_i\Lambda^{r_i}$ one finds
\[
  B_2=u_2+u_3,\quad B_3=u_3+u_4+\tfrac12u_2^{2},\quad
  B_4=u_4+u_5-u_2u_3,\quad B_5=u_5+\tfrac12u_3^{2},
\]
and a direct computation gives $\Jac_u(B_2,\dots,B_5)=1=\Lambda^{\,m-1}$, so $\detJ{G}=1$ as it must be. Here $\Lambda$ is constant, so the hypothesis of \cref{thm:autononinj} fails, and indeed $G$ is an automorphism. More generally, when $\Lambda$ is constant \cref{eq:master} reduces to $\Jac_u(B_2,\dots,B_n)=\kappa$, which is the Keller condition for $\Gbar$ alone: the graded problem in dimension $n$ becomes the ordinary Jacobian Conjecture in dimension $n-1$, with no contracted locus and no degeneracy. The contracted line of \cref{eq:map}, and not the hyperbolic signature, is what makes a counterexample possible.
\end{rem}

\subsection{The parameter variety}

Fix $r,s$ and a bound $\bd=(d_1,d_2,d_3)$, and set
\[
  \VAd=\VA\cap\C[u,v]_{\le d_1},\quad
  \VBd=\VB\cap\C[u,v]_{\le d_2},\quad
  \VLd=\C[u,v]_{\le d_3},
\]
all three $\Q$-rational linear subspaces. Let
\[
  \mu(A,B,\Lambda)=\Lambda\jb{B}{A}+sA\jb{B}{\Lambda}+rB\jb{\Lambda}{A}
  \;\in\;\C[u,v]_{\le D},
\]
and $D=d_1+d_2+d_3-2$.
Each term of $\mu$ is linear in each of $A$, $B$, $\Lambda$ separately: $\mu$ is trilinear. Writing 
\[
\mu=\sum_{i,j}\mu_{ij}(A,B,\Lambda)u^iv^j,
\]
 each $\mu_{ij}$ is a form of tridegree $(1,1,1)$ with coefficients in $\Z[r,s]$. We define
\[
  \Ypar\;\subseteq\;
  \bP\bigl(\VAd\bigr)\times\bP\bigl(\VBd\bigr)
  \times\bP\bigl(\VLd\bigr)
\]
to be the closed subscheme cut out by the $\binom{D+2}{2}-1$ multihomogeneous equations $\mu_{ij}=0$,   $1\le i+j\le D$, and $\Yparo$ to be the open subscheme on which $\mu_{00}\neq0$ and $\deg\Lambda\ge1$. Thus $\Ypar$ is a multilinear section of a product of projective spaces, defined over $\Z$.

\begin{theorem}\label{thm:points}
Let $k$ be a field of characteristic zero. There is a bijection between $\Yparo(k)$ and the set of Keller maps over $k$ of the form \cref{eq:triple}, graded with source weights $\bw=(1,-r,-s)$ and target weights $\bq=(-s,-r,1)$, which are not polynomial automorphisms, with $\deg A\le d_1$, $\deg B\le d_2$, $\deg\Lambda\le d_3$, taken modulo the scaling action $(A,B,\Lambda)\mapsto(\alpha A,\beta B,\gamma\Lambda)$ of $\Gm^{3}$.
\end{theorem}

\begin{proof}
A $k$-point of a projective space over a field has $k$-rational homogeneous coordinates, so a $k$-point of the triple product lifts to a triple $(A,B,\Lambda)$ with coefficients in $k$, uniquely up to $\Gm^{3}(k)$. The conditions $\mu_{ij}=0$ for $(i,j)\neq(0,0)$ say that $\mu$ is the constant $\kappa=\mu_{00}$, and $\mu_{00}\neq0$ says $\kappa\in k^{\times}$; the triple then defines a Keller map by \cref{thm:master}, polynomial by the conditions $A\in\VA$, $B\in\VB$, and graded with the stated weights because \cref{eq:triple} is the description of that graded piece. It is not a polynomial automorphism: the degree $\deg G=[\Omega:K]$ is unchanged by base change to $\bar k$, so \cref{thm:autononinj} applies there and gives $\deg G\ge2$, which forbids a polynomial inverse over $k$. Conversely such a map determines its triple. Under $(\alpha,\beta,\gamma)$ one has $\kappa\mapsto\alpha\beta\gamma\kappa$, so $\kappa$ may be normalised to $1$ over $k$ and is not a modulus.
\end{proof}

The map $F$ of \cref{eq:map} has $\deg_{u,v}A=4$, $\deg_{u,v}B=3$ and $\deg_{u,v}\Lambda=1$, so it defines a point
\[
  [F]\;\in\;\Yparo[1,2,(4,3,1)](\Q).
\]
The function $[G]\mapsto\deg\Gbar=[\C(u,v):\C(P,Q)]$ is constructible on $\Yparo[]$ and stratifies it; by \cref{thm:main} the point $[F]$ lies in the stratum $\deg=3$.

\begin{cor}
Let $k$ be a field of characteristic zero, and fix $r,s>0$ and $\bd=(d_1,d_2,d_3)$. Then there exists a Keller map over $k$ of the form \cref{eq:triple}, graded with source weights $\bw=(1,-r,-s)$ and target weights $\bq=(-s,-r,1)$, with $\deg A\le d_1$, $\deg B\le d_2$, $\deg\Lambda\le d_3$, which is not a polynomial automorphism, if and only if
\[
  \Yparo(k)\neq\emptyset .
\]
The Jacobian Conjecture, restricted to graded maps of this weight type and of bounded degree, is therefore the statement that an explicit multiprojective scheme over $\Z$ has no $k$-point.
\end{cor}

\begin{rem}
The two fields behave differently. For $k=\C$, or any algebraically closed field of characteristic zero, the emptiness of $\Yparo$ is decided by elimination in finitely many steps, so for each fixed $(r,s,\bd)$ the graded Jacobian Conjecture is a finite computation, and the rank criterion of the next subsection is the practical form of it. 

For $k=\Q$ the same question is a Diophantine problem on a scheme over $\Z$, and no such procedure is known. This is why we later pass to the weighted height on $\Yparo$: it is the right invariant for counting the $\Q$-points, and not for deciding whether they exist.
\end{rem}

\subsection{A determinantal fibration}

For fixed $B$ and $\Lambda$ the expression $\mu$ is linear in $A$. This gives $\Ypar$ a fibration by linear spaces over a determinantal base, and turns the search for solutions into linear algebra.

\begin{prop}\label{prop:determinantal}
For $(B,\Lambda)\in\VBd\times\VLd$ let
\[
  \TBL\colon \VAd\longrightarrow \C[u,v]_{\le D},
  \qquad \TBL(A)=\mu(A,B,\Lambda),
\]
a linear map whose matrix entries are bilinear forms in the coefficients of $B$ and $\Lambda$, and let
\[
  \pi\colon \C[u,v]_{\le D}\longrightarrow \C[u,v]_{\le D}/\C
\]
denote the quotient by the constants. Rescaling $B$ or $\Lambda$ multiplies $\TBL$ by a nonzero scalar, so the kernels and ranks below depend only on the classes $[B]$ and $[\Lambda]$. The projection
\[
  \Ypar\longrightarrow
  \bP\bigl(\VBd\bigr)\times\bP\bigl(\VLd\bigr)
\]
has fibre $\bP\bigl(\ker(\pi\circ \TBL)\bigr)$ over $([B],[\Lambda])$. If moreover $\deg\Lambda\ge1$, this fibre meets $\Yparo$ if and only if
\[
  \rank \TBL
  =\rank(\pi\circ \TBL)+1 .
\]
In particular the locus of $([B],[\Lambda])$ with $\deg\Lambda\ge1$ admitting a solution is constructible, defined by the vanishing and the non-vanishing of minors of matrices bilinear in $(B,\Lambda)$.
\end{prop}

\begin{proof}
The fibre is by definition the projectivisation of $\{A:\mu(A,B,\Lambda)\in\C\}=\ker(\pi\circ \TBL)$. Restricted to that kernel, $\TBL$ takes values in $\C$, so its image there is either $0$ or all of $\C$; the fibre contains a point with $\kappa\neq0$ exactly in the second case, that is exactly when $\ker(\pi\circ \TBL)\not\subseteq\ker \TBL$. By rank--nullity,
\[
  \dim\ker(\pi\circ \TBL)-\dim\ker \TBL
  =\rank \TBL-\rank(\pi\circ \TBL)\in\{0,1\},
\]
so that inclusion fails precisely when the difference of ranks equals $1$. Since $\deg\Lambda\ge1$ by hypothesis, $\kappa\neq0$ is the only remaining condition defining $\Yparo$, which gives the stated equivalence. The entries of $\TBL$ in the monomial bases are the coefficients of $\mu(u^iv^j,B,\Lambda)$, bilinear in $(B,\Lambda)$ by trilinearity of $\mu$.
\end{proof}

The hypothesis $\deg\Lambda\ge1$ cannot be dropped. If $\Lambda$ is a nonzero constant then $\jb{B}{\Lambda}=\jb{\Lambda}{A}=0$, so \cref{eq:master} reduces to
\[
  \jb{B}{A}=\kappa/\Lambda ,
\]
the Keller condition for the pair $(A,B)$ on $\A^2_{u,v}$ alone, as in \cref{rem:vasyunin-descent}. That stratum is nonempty and contains no counterexample: already $A=v$, $B=u$, $\Lambda=1$ satisfies it, with $G=(z,y,x)$ the reversal of the coordinates. Searching it would be searching the Jacobian Conjecture in dimension two.

\Cref{prop:determinantal} is the practical form of the recipe of \cref{rem:recipe}: to search for graded counterexamples of bounded degree one runs over the coefficients of $B$ and $\Lambda$ with $\deg\Lambda\ge1$ and imposes a rank jump, rather than computing a Gr\"obner basis of the full trilinear system.

\subsection{Excess intersection}

The system \cref{eq:master} is heavily overdetermined at the degrees of \cref{eq:map}. For $(r,s)=(1,2)$ and $\bd=(4,3,1)$ one has
\[
  \dim \VAd=13,\qquad \dim \VBd=9,\qquad \dim \VLd=3,
\]
since $\C[u,v]_{\le4}$ has dimension $15$ and $\VA$ omits from it only the monomials $1$ and $u$, while $\C[u,v]_{\le3}$ has dimension $10$ and $\VB$ omits only $1$. The ambient product $\bP^{12}\times\bP^{8}\times\bP^{2}$ therefore has dimension $22$, while $D=6$ and the number of equations is $\binom{8}{2}-1=27$. A generic system of $27$ forms of tridegree $(1,1,1)$ on a $22$-dimensional variety has empty intersection; the expected dimension is $-5$. The existence of \cref{eq:map} is therefore an excess intersection.

The discrepancy is in fact larger, because $\Yparo$ is not a bare intersection: the group $\Gaut$ described below acts on it, and for $\bd=(4,3,1)$ the subgroup preserving the degree bounds is the linear one, of dimension $3$. A nonempty $\Yparo$ thus contains an entire orbit, and the gap from the expected dimension is at least $5$ together with the dimension of that orbit. The sharp form of the question left open in \cref{sec:9} is the following.

\begin{rem}\label{rem:orbit}
Is $\Yparo[1,2,(4,3,1)]$ a single orbit of the group of liftable automorphisms? More generally, for which signatures $(r,s)$ and bounds $\bd$ is $\Yparo$ nonempty, and when is it positive-dimensional modulo that group?
\end{rem}

The group in question is the one induced on $\A^2_{u,v}$ by the graded automorphisms of $\A^3_{\bw}$, that is by the equivariant automorphisms of $\C^3$ in the sense of \cref{sec-2}. Such an automorphism has target weights equal to its source weights, so it takes the form $x\mapsto x\phi$, $y\mapsto x^{-r}\psi$, $z\mapsto x^{-s}\chi$ with $\phi,\psi,\chi$ invariant. By the argument in the second paragraph of the proof of \cref{thm:autononinj}, it descends to an automorphism of $\A^2_{u,v}$, whose Jacobian is a nonzero constant; and by \cref{rem:exponent}, with $n=3$ and $m=r+s$,
\[
  \Jacuv\bigl(\psi\phi^{r},\,\chi\phi^{s}\bigr)=\kappa'\phi^{\,r+s-1} .
\]
The left-hand side is constant and $r+s-1\ge1$, so $\phi$ is constant. The induced subgroup of $\Aut(\A^2)$ is therefore contained in
\[
  \Gaut=\bigl\{\varphi\in\Aut(\A^2):
    \varphi(O)=O,\ \varphi^{*}\VA\subseteq\VA,\ \varphi^{*}\VB\subseteq\VB
  \bigr\},
\]
and in fact equals it: given such a $\varphi$, the monomials $u$ and $v$ lie in $\VB$ and $\VA$ respectively, so $\varphi_1=\varphi^{*}u\in\VB$ and $\varphi_2=\varphi^{*}v\in\VA$, and
\[
  \tilde\varphi(x,y,z)=\bigl(x,\;x^{-r}\varphi_1(u,v),\;x^{-s}\varphi_2(u,v)\bigr)
\]
is a polynomial equivariant map inducing $\varphi$, invertible because $\varphi$ is. For $(r,s)=(1,2)$ the conditions read $\varphi(O)=O$ together with $\varphi^{*}(u^{2},v)\subseteq(u^{2},v)$, and since $\varphi_1^{2}$ lies in $(u^2,v)$ automatically, they amount to $\partial_u\varphi_2(O)=0$: $\Gaut$ is the group of automorphisms fixing $O$ for which $\partial_u$ is an eigenvector of the differential at $O$.


\subsection{Weighted heights on the parameter variety}

The tori of both actions act on the parameter space, and their quotient is of the positive-weight type of \cite{sparsity}.

Let $\delta_{\lambda,\rho,\varsigma}(x,y,z)=(\lambda x,\rho y,\varsigma z)$ and let $\epsilon_{\alpha,\beta,\gamma}$ be the analogous automorphism of the target; both are equivariant, and $G\mapsto\epsilon\circ G\circ\delta^{-1}$ acts on triples by
\[
\begin{split}
  A_{ij} 		&	\mapsto\alpha\lambda^{\,s-ri-sj}\rho^{-i}\varsigma^{-j}A_{ij},	\\
  B_{ij}		&	\mapsto\beta\lambda^{\,r-ri-sj}\rho^{-i}\varsigma^{-j}B_{ij},	\\
  \Lambda_{ij}	&	\mapsto\gamma\lambda^{-1-ri-sj}\rho^{-i}\varsigma^{-j}\Lambda_{ij},
\end{split}  
\]
the subscripts indexing the coefficient of $u^iv^j$. The one-parameter subgroup $\lambda=1$, $\rho=\varsigma=t^{-1}$, $\alpha=\beta=\gamma=t$ acts on every coefficient by $t^{\,1+i+j}$.

\begin{prop}\label{prop:wps}
Let 
\[
N=\dim \VAd+\dim \VBd+\dim \VLd
\]
 and let $\bomega\in\Z_{>0}^{N}$ be the vector whose entry at the coordinate indexing the coefficient of $u^iv^j$ equals $1+i+j$. Then the $\Gm$-action above identifies
\[
  \bigl(\VAd\oplus \VBd\oplus \VLd\bigr)
  \setminus\{0\}\;\big/\;\Gm \;=\;\WP(\bomega),
\]
a weighted projective space, well-formed when $d_1\ge2$ and $d_3\ge1$, and the equation $\mu_{ij}=0$ is homogeneous for it of weighted degree $5+i+j$. The graded Keller counterexamples of bounded degree are therefore parametrised by a locally closed subvariety $\Wpar\subseteq\WP(\bomega)$ defined over $\Z$.
\end{prop}

\begin{proof}
All the exponents $1+i+j$ are positive, so the quotient of the punctured affine space by the displayed $\Gm$-action is by definition the weighted projective space $\WP(\bomega)$. The coefficient $\Lambda_{00}$ occurs, since $\VLd=\C[u,v]_{\le d_3}$ contains the constants, and carries weight $1$; hence the action is faithful and no reduction of the weights is needed. For well-formedness one must check that any $N-1$ of the weights have greatest common divisor $1$. A subset containing the coordinate $\Lambda_{00}$ has $\gcd=1$ at once. Its complement contains $\Lambda_{10}$, of weight $2$, when $d_3\ge1$, and $v^{2}\in\VAd$, of weight $3$, when $d_1\ge2$; these are coprime.

For the homogeneity, write $A=\sum A_{ij}u^iv^j$. Under the subgroup,
\[
  A'(u,v)=\sum t^{\,1+i+j}A_{ij}u^iv^j=t\,A(tu,tv),
\]
and likewise $B'(u,v)=t\,B(tu,tv)$ and $\Lambda'(u,v)=t\,\Lambda(tu,tv)$. Differentiating, $\partial_u A'=t^{2}A_u(tu,tv)$ and similarly for the other first derivatives, so each bracket satisfies $\jb{B'}{A'}=t^{4}\jb{B}{A}(tu,tv)$. Each of the three terms of $\mu$ is a bracket multiplied by one of $A,B,\Lambda$, so
\[
  \mu(A',B',\Lambda')(u,v)=t^{5}\,\mu(A,B,\Lambda)(tu,tv),
\]
and comparing coefficients of $u^iv^j$ gives $\mu_{ij}\mapsto t^{\,5+i+j}\mu_{ij}$. The locus $\Wpar$ is then cut out by these homogeneous equations together with the open conditions $\mu_{00}\neq0$ and $\deg\Lambda\ge1$, all defined over $\Z$.
\end{proof}

The quotient here is not the one of \cref{thm:points}. Both are quotients of the affine cone of triples satisfying \cref{eq:master}, but $\Ypar$ divides by the group $\Gm^{3}$ rescaling $A$, $B$ and $\Lambda$ separately, whereas $\WP(\bomega)$ divides by the single $\Gm$ above, which acts on the coefficient of $u^iv^j$ by $t^{\,1+i+j}$ and is therefore not a subgroup of $\Gm^{3}$. For $\bd=(4,3,1)$ the two ambient spaces have dimensions $22$ and $24$. A point of $\Wpar$ is thus a triple taken modulo less than a point of $\Yparo$, and the counting function below counts these finer objects; passing to genuine equivalence classes of maps would require dividing further, by the remaining torus and by $\Gaut$.

The weighted height of \cite{bgsh,sparsity} is nevertheless the natural height on $\WP(\bomega)$, and one may ask for the asymptotics of
\[
  \Ncount(X)=\#\bigl\{[G]\in\Wpar(\Q):
     \hwt([G])\le X\bigr\}.
\]
Unlike the counting function $Z(X)$ of \cref{sec:9}, this one is well posed: $\WP(\bomega)$ is proper and the weights $\bomega$ are positive, so no height theory for hyperbolic quotients is required. The hyperbolic example is counted by an elliptic parameter space. If \cref{rem:orbit} has the answer that $\Yparo$ is a single $\Gaut$-orbit, then $\Ncount(X)$ measures the growth of that orbit together with the fibres of $\Wpar\to\Yparo[]$, and nothing more; if not, its exponent should again come from a linear program, as in the positive-weight setting.

\subsection{Field of moduli}

Since $\Yparo$ and $\Gaut$ are defined over $\Q$, the quotient carries a field of moduli for each graded counterexample, and one may ask whether it is a field of definition. It is convenient to work with the group upstairs. Let $\Gaut^{\sharp}$ denote the group of pairs $(\delta,\epsilon)$ of equivariant polynomial automorphisms of the source and the target $\C^3$, acting on maps by $G\mapsto\epsilon\circ G\circ\delta^{-1}$, and let $\Aut^{\sharp}(G)$ be the stabiliser of $G$. Write $\Gamma=\Gal(\Qbar/\Q)$. The obstruction to descending $G$ from a field of definition to its field of moduli lies in $H^{1}\bigl(\Gamma,\Aut^{\sharp}(G)\bigr)$.

Here the grading intervenes. By \cref{eq:equiv} the pair $(\sigt,\taut)$ satisfies $\taut\circ G\circ\sigt^{-1}=G$ for every $t$, so $t\mapsto(\sigt,\taut)$ embeds $\Gm$ in $\Aut^{\sharp}(G)$, for every graded $G$ at once. This $\Gm$ acts trivially on $\A^2_{u,v}$, since $u$ and $v$ are invariants; it is the kernel of the homomorphism $\Gaut^{\sharp}\to\Gaut$ of \cref{sec-8}, and by the computation there, which forces $\phi$ to be constant, that homomorphism is surjective with $\Gaut=\Gaut^{\sharp}/\Gm$. From the exact sequence
\[
  1\longrightarrow\Gm\longrightarrow\Aut^{\sharp}(G)\longrightarrow\pi_0\longrightarrow1,
\]
where $\pi_0$ is the image of $\Aut^{\sharp}(G)$ in $\Gaut$, one obtains the exact sequence of pointed sets
\[
  H^{1}(\Gamma,\Gm)\longrightarrow H^{1}\bigl(\Gamma,\Aut^{\sharp}(G)\bigr)
  \longrightarrow H^{1}(\Gamma,\pi_0),
\]
and $H^{1}(\Gamma,\Gm)=1$ by Hilbert 90. Exactness at the middle term says that the fibre over the neutral class is trivial; hence a nontrivial obstruction class must have nontrivial image in $H^{1}(\Gamma,\pi_0)$. The connected part contributes nothing, and any obstruction is detected by the component group.

\begin{rem}\label{rem:aut}
Compute $\Aut^{\sharp}(F)$ for the map \cref{eq:map}. Is it connected, that is, is $\pi_0$ trivial? If not, the field of moduli of a graded counterexample need not be a field of definition, and the obstruction is a Kummer class of the same kind as the one appearing on the stacky stratum in \cref{rem:kummer}, now at the level of moduli rather than of fibres. One may also ask whether every point of $\Yparo(\C)$ is defined over $\Qbar$; this holds whenever $\Yparo[]$ is a single $\Gaut$-orbit of a $\Qbar$-point, and would say that a graded counterexample can carry no transcendental modulus.
\end{rem}

\section{Higher dimensions}
\label{sec:higher}

The construction of \cref{sec-8} is not special to dimension three. We record here that it generalises to $\C^n$ for every $n\ge3$, provided the weight vector has exactly one positive entry and that entry equals $1$. The master equation stays multilinear, and the parameter variety becomes a multilinear section of a product of $n$ projective spaces. We also record where the generalisation stops, which sharpens the trichotomy of \cref{sec:9}.

Fix $n\ge3$ and a tuple $\br=(r_2,\dots,r_n)$ of positive integers. Write $\bw=(1,-r_2,\dots,-r_n)$, use $x,y_2,\dots,y_n$ for the coordinates on the source, and set
\[
  m=\sum_{i=2}^{n}r_i .
\]
By \cref{cor:keller-weights} the target weights of a graded Keller map are a permutation of $\bw$; relabelling the target coordinates, a graded isomorphism by \cref{cor:keller-weights}(ii), we take $\bq=\bw$ and write $a_1,\dots,a_n$ for the target coordinates. For $n=3$ this is the identity type of \cref{rem:exponent} rather than the reversal of \cref{sec-8}, a difference of labelling only.

For a multi-index $\bb=(b_2,\dots,b_n)$ we write $u^{\bb}=u_2^{b_2}\cdots u_n^{b_n}$ and $|\bb|=\sum_i b_i$.

\subsection{The graded pieces}

A monomial $x^{c}y_2^{b_2}\cdots y_n^{b_n}$ has weight $c-\sum_i r_ib_i$, so it is invariant exactly when $c=\sum_i r_ib_i$, and then it equals $u^{\bb}$ for
\[
  u_i=x^{r_i}y_i .
\]
Hence $\C[x,y]^{\sigma}=\C[u_2,\dots,u_n]$, freely. The same count gives the pieces we need. A monomial of weight $1$ has $c=1+\sum_ir_ib_i\ge1$, so it is $x$ times an invariant, and the weight-$1$ piece is $x\,\C[u]$. A monomial of weight $-r_j$ has $c=\sum_ir_ib_i-r_j$, which is nonnegative exactly when $\sum_ir_ib_i\ge r_j$, so the weight-$(-r_j)$ piece is $x^{-r_j}V_j$ with
\[
  V_j:=\Span\Bigl\{\,u^{\bb}\ :\ \sum_{i} r_ib_i\ge r_j\,\Bigr\},
\]
a monomial ideal of $\C[u]$. For $n=3$, with $(r_2,r_3)=(r,s)$, these are $V_2=\VB$ and $V_3=\VA$.

A graded polynomial map $G\colon\C^n\to\C^n$ of this weight type is therefore exactly a tuple
\begin{equation}\label{eq:ntuple}
  G=\bigl(x\Lambda,\;x^{-r_2}B_2,\;\dots,\;x^{-r_n}B_n\bigr),
  \qquad \Lambda\in\C[u],\quad B_j\in V_j .
\end{equation}
On the target, $\C[a]^{\tau}=\C[P_2,\dots,P_n]$ freely, with $P_j=a_1^{r_j}a_j$, and $G^{*}P_j=B_j\Lambda^{r_j}$; so $G$ descends to
\[
  \Gbar=\bigl(B_2\Lambda^{r_2},\dots,B_n\Lambda^{r_n}\bigr)
  \colon \A^{n-1}_{u}\longrightarrow \A^{n-1}_{P} .
\]

\subsection{The master equation}

\begin{prop}\label{prop:ndescent}
For $G$ as in \cref{eq:ntuple}, up to a universal sign,
\[
  \detJ{G}=\pm\,\frac{\Jac_{u}(P_2,\dots,P_n)}{\Lambda^{\,m-1}} .
\]
\end{prop}

\begin{proof}
On $\{x\neq0\}$ use coordinates $(x,u_2,\dots,u_n)$. The change $(x,y)\mapsto(x,u)$ has triangular Jacobian matrix with diagonal $(1,x^{r_2},\dots,x^{r_n})$, hence determinant $\pm x^{m}$; likewise $(a_1,\dots,a_n)\mapsto(a_1,P_2,\dots,P_n)$ has determinant $\pm a_1^{m}$. In these coordinates $G$ is
\[
  (x,u)\longmapsto\bigl(x\Lambda(u),\,P_2(u),\dots,P_n(u)\bigr),
\]
whose Jacobian matrix is block triangular with determinant $\Lambda\Jac_{u}(P_2,\dots,P_n)$. Since $a_1=x\Lambda$ along $G$, the chain rule gives
\[
  \Lambda\,\Jac_{u}(P_2,\dots,P_n)
  =\pm\,(x\Lambda)^{m}\,\detJ{G}\,x^{-m}
  =\pm\,\Lambda^{m}\,\detJ{G}
\]
on a dense open set, hence everywhere, and the claim follows.
\end{proof}

This proves the formula asserted in \cref{rem:exponent}. The next result is the general form of \cref{thm:master}.

\begin{theorem}\label{thm:nmaster}
Let $\Lambda,B_2,\dots,B_n\in\C[u_2,\dots,u_n]$ and set
\[
  \mu(B_2,\dots,B_n,\Lambda)
  =\Lambda\,\Jac_{u}(B_2,\dots,B_n)
   +\sum_{i=2}^{n}r_iB_i\,\Jac_{u}(B_2,\dots,\Lambda,\dots,B_n),
\]
the $i$-th summand having $\Lambda$ in the $i$-th slot. Then
\[
  \Jac_{u}\bigl(B_2\Lambda^{r_2},\dots,B_n\Lambda^{r_n}\bigr)
  =\Lambda^{\,m-1}\,\mu(B_2,\dots,B_n,\Lambda),
\]
and consequently $G$ as in \cref{eq:ntuple} is a Keller map if and only if
\begin{equation}\label{eq:nmaster}
  \mu(B_2,\dots,B_n,\Lambda)=\kappa \qquad\text{for some }\kappa\in\C^{\times}.
\end{equation}
Moreover $\mu$ is multilinear, of multidegree $(1,1,\dots,1)$ in its $n$ arguments $B_2,\dots,B_n,\Lambda$.
\end{theorem}

\begin{proof}
Differentiating $P_i=B_i\Lambda^{r_i}$ gives
\[
  \frac{\partial P_i}{\partial u_j}
  =\Lambda^{\,r_i-1}\Bigl(\Lambda\frac{\partial B_i}{\partial u_j}
   +r_iB_i\frac{\partial\Lambda}{\partial u_j}\Bigr),
\]
so the $i$-th row of the matrix $J_u(P)$ is $\Lambda^{\,r_i-1}\bigl(\Lambda\nabla B_i+r_iB_i\nabla\Lambda\bigr)$. Factoring $\Lambda^{\,r_i-1}$ out of each of the $n-1$ rows contributes $\Lambda^{\,m-(n-1)}$. Expanding the remaining determinant by multilinearity in the rows gives $2^{\,n-1}$ terms, one for each choice of $\Lambda\nabla B_i$ or $r_iB_i\nabla\Lambda$ in each row. Any term choosing $r_iB_i\nabla\Lambda$ in two or more rows has two proportional rows and vanishes. The term choosing $\Lambda\nabla B_i$ in every row is $\Lambda^{\,n-1}\Jac_{u}(B_2,\dots,B_n)$, and the term choosing $r_iB_i\nabla\Lambda$ in the single row $i$ is $\Lambda^{\,n-2}r_iB_i\Jac_{u}(B_2,\dots,\Lambda,\dots,B_n)$. Hence
\[
  \Jac_{u}(P)=\Lambda^{\,m-(n-1)}\cdot\Lambda^{\,n-2}\,\mu=\Lambda^{\,m-1}\,\mu ,
\]
and \cref{eq:nmaster} follows from \cref{prop:ndescent}. For the multidegree, the first summand is linear in $\Lambda$ and, a determinant being multilinear in its rows, in each $B_i$. In the $i$-th summand, $B_i$ occurs once as the outer factor and not inside the determinant, $\Lambda$ occurs once inside it, and each $B_j$ with $j\neq i$ occurs once as a row.
\end{proof}

For $n=3$, writing $B_2=B$, $B_3=A$ and $(r_2,r_3)=(r,s)$, one has $\Jac_{u}(B_2,B_3)=\jb{B}{A}$ and \cref{eq:nmaster} reads
\[
  \Lambda\jb{B}{A}+rB\jb{\Lambda}{A}+sA\jb{B}{\Lambda}=\kappa,
\]
which is \cref{eq:master}. The map of \cref{rem:vasyunin} is the case $n=5$, $\br=(1,2,3,4)$, $\Lambda=1$, computed in \cref{rem:vasyunin-descent}.

\subsection{Consequences}

\begin{lem}\label{lem:nbasept}
Let $G$ be as in \cref{eq:ntuple} and suppose \cref{eq:nmaster} holds. Let $O$ be the origin of $\A^{n-1}_u$. Then $B_i(O)=0$ for every $i$, $\Lambda(O)\neq0$, and $dB_2\wedge\dots\wedge dB_n|_{O}\neq0$.
\end{lem}

\begin{proof}
Every monomial $u^{\bb}$ spanning $V_i$ satisfies $\sum_jr_jb_j\ge r_i>0$, so $|\bb|\ge1$; hence $V_i$ contains no constant and lies in $(u_2,\dots,u_n)$, giving $B_i(O)=0$. Evaluating \cref{eq:nmaster} at $O$ kills every summand of the sum, leaving
\[
  \Lambda(O)\,\Jac_{u}(B_2,\dots,B_n)(O)=\kappa\neq0 .
\]
\end{proof}

\begin{theorem}\label{thm:nnoninj}
Let $G$ be as in \cref{eq:ntuple}, satisfying \cref{eq:nmaster}, with $\Lambda\notin\C$. Then $G$ is a counterexample to the Jacobian Conjecture.
\end{theorem}

\begin{proof}
The argument of \cref{thm:autononinj} applies verbatim, Zariski's main theorem and the unit argument holding in every dimension: were $G$ injective it would be a polynomial automorphism, its inverse would again be graded, and $\Gbar$ would be an automorphism of $\A^{n-1}$, so $\Jac_{u}(P)$ would be a nonzero constant. By \cref{thm:nmaster} that constant is $\kappa\Lambda^{\,m-1}$, and $m\ge n-1\ge2$, so $\Lambda$ would be constant.
\end{proof}

The hypothesis $n\ge3$ enters only through $m\ge2$. For $n=2$ and $r_2=1$ the exponent $m-1$ vanishes and the argument gives nothing, consistently with \cref{thm:plane}, where the hyperbolic case in dimension two admits no counterexample.

Finally, fix a bound $\bd=(d_2,\dots,d_n,d_{\Lambda})$ and set
\[
  V_i(\bd)=V_i\cap\C[u]_{\le d_i},\qquad
  \VL(\bd)=\C[u]_{\le d_{\Lambda}},\qquad
  D=d_{\Lambda}+\sum_{i\ge2}d_i-(n-1).
\]
Writing $\mu=\sum_{\bb}\mu_{\bb}\,u^{\bb}$, each $\mu_{\bb}$ is a form of multidegree $(1,\dots,1)$ with coefficients in $\Z[r_2,\dots,r_n]$, by \cref{thm:nmaster}. The graded Keller counterexamples of this weight type and bounded degree are therefore the points of the open subscheme
\[
  \Yparo[\br,\bd]\;\subseteq\;
  \bP\bigl(V_2(\bd)\bigr)\times\dots\times\bP\bigl(V_n(\bd)\bigr)
  \times\bP\bigl(\VL(\bd)\bigr)
\]
on which $\mu_{\mathbf 0}\neq0$ and $\deg\Lambda\ge1$, cut out in the ambient product by the $\binom{D+n-1}{n-1}-1$ equations $\mu_{\bb}=0$ with $1\le|\bb|\le D$, and defined over $\Z$. The proofs of \cref{thm:points}, \cref{prop:determinantal} and \cref{prop:wps} use only the multilinearity of $\mu$ and go through unchanged: fixing all arguments but one leaves a linear map whose rank jump detects $\kappa\neq0$, and the one-parameter subgroup acting on the coefficient of $u^{\bb}$ by $t^{\,1+|\bb|}$ presents the whole space as a weighted projective space $\WP(\bomega)$ with positive weights $1+|\bb|$, well-formed under a condition on $\bd$ checked as in \cref{prop:wps}. The computation there generalises to
\[
  \mu_{\bb}\longmapsto t^{\,2n-1+|\bb|}\mu_{\bb},
\]
which is the weighted degree $5+|\bb|$ of \cref{prop:wps} when $n=3$. For $n=3$ all of this recovers \cref{sec-8}.

\subsection{Where the construction stops}

Two hypotheses are doing work, and neither can be dropped by the methods above.

The positive weight must equal $1$. It is this that makes the invariants $u_i=x^{r_i}y_i$ generate $\C[x,y]^{\sigma}$ freely, so that the quotient is again an affine space and $\Jac_{u}$ has its usual meaning. Already for $\bw=(2,-1,-1)$ the invariant ring is $\C[xy^{2},xyz,xz^{2}]$, a quadric cone, as noted in \cref{rem:exponent}; the graded pieces are then modules over a non-polynomial ring rather than monomial ideals, and the passage from \cref{eq:ntuple} to a linear-algebra problem fails at the first step.

There must be exactly one positive weight. This is what makes the weight-$1$ piece equal to $x\,\C[u]$, which produces the factor $\Lambda$ and with it the contracted locus $\{\Lambda=0\}$, the object around which all of \cref{sec-8} is organised. With two positive weights, say $\bw=(1,1,-1)$, the invariant ring $\C[xz,yz]$ is still free, so singularity of the quotient is not the obstruction; but the weight-$1$ piece contains both $x$ and $y$, and $y$ is not $x$ times a polynomial invariant. We do not know what replaces \cref{eq:ntuple} in that case.

Within the hyperbolic case the reduction is therefore governed by more than the signature: by the number of positive weights, and by the value of the positive one. A parameter variety for the remaining hyperbolic signatures would require different methods.


\section{Classification of graded Keller maps $\A^n_{\bw}\to\A^n_{\bq}$}
\label{sec:classify}

We now ask the question the previous sections have been preparing: for which data does a graded Keller counterexample exist? The data is a dimension, a weight vector, and a bound on the degrees, and the answer is organised as a descent through that data. At each stage we say what is settled and what is not; \cref{tab:classify} collects the outcome.

\subsection{The classification datum}
\label{subsec:datum}

Let $G\colon\A^n_{\bw}\to\A^n_{\bq}$ be a graded Keller map. By \cref{cor:keller-weights} the vector $\bq$ is a permutation of $\bw$, say $\bq=\pi(\bw)$. The permutation is not part of the datum. Indeed, by \cref{cor:keller-weights}(ii) the coordinate permutation $\pi$ is a graded isomorphism $\A^n_{\pi(\bw)}\to\A^n_{\bw}$, and composing $G$ with it changes neither the Keller property nor the property of being a polynomial automorphism; so the families attached to two permutations are identified. We may therefore always normalise $\bq=\bw$, as we did in \cref{eq:goodtype} below, and the reversal convention of \cref{sec-8} and the identity convention of \cref{sec:higher} differ only by this identification.

Three further reductions apply. We assume $\bw\ne 0$, the case $\bw=0$ imposing no condition at all and being the Jacobian Conjecture itself in dimension $n$. By \cref{cor:keller-weights}(iv) we may normalise $\gcd(w_1,\dots,w_n)=1$. Replacing $t$ by $t^{-1}$ replaces $\bw$ by $-\bw$ and changes neither $G$ nor its being graded, so we may fix the sign of any one nonzero weight. Finally, a permutation of the source coordinates carries $\A^n_{\bw}$ to $\A^n_{\bw'}$ for the correspondingly permuted $\bw'$, so only the multiset of weights matters.

Write
\[
  \Kfam(n,\bw)
\]
for the set of graded Keller maps $\A^n_{\bw}\to\A^n_{\bw}$ which are not polynomial automorphisms, taken up to the equivalence generated by the graded automorphisms of source and target; that is, up to the group $\Gaut^{\sharp}$ of \cref{sec-8}, whose image in $\Aut(\A^{n-1})$ is the group $\Gaut$ used in the parameter space. The classification problem is to decide when $\Kfam(n,\bw)$ is empty, and to describe it when it is not.

\subsection{First cut: signature and dimension}

\begin{theorem}\label{thm:firstcut}
Let $\bw\in\Z^n\setminus\{0\}$ be a weight vector, normalised as above.
\begin{enumerate}
\item If $\bw$ is elliptic, then $\Kfam(n,\bw)=\emptyset$, for every $n$.
\item If $n\le 2$, then $\Kfam(n,\bw)=\emptyset$, for every signature.
\end{enumerate}
\end{theorem}

\begin{proof}
(1) is \cref{thm:rigid}. For (2) with $n=2$ this is \cref{thm:plane}. For $n=1$ the normalisation gives $\bw=(\pm1)$; a Keller map $\A^1\to\A^1$ has constant nonzero derivative, hence is affine, and gradedness with $w_1\neq0$ forces $G(0)=0$ by \cref{lem:grading}(ii), so $G$ is linear and invertible.
\end{proof}

What remains is $n\ge3$ with $\bw$ parabolic or hyperbolic.

\subsection{Second cut: the parabolic case}
\label{subsec:parabolic}

Let $\bw$ be parabolic, let $Z=\{i:w_i=0\}$ and $z=\#Z\ge1$, and let $\bw'$ denote the nonzero part, of length $n-z\ge1$. Since $\bq$ is a permutation of $\bw$, the target has $z$ zero weights as well. Two cases occur, according to the signature of $\bw'$, and both reduce to problems already named.

\begin{theorem}\label{thm:parabolic-elliptic}
Suppose the nonzero weights of $\bw$ are all of the same sign. Then $\Kfam(n,\bw)\neq\emptyset$ if and only if the Jacobian Conjecture fails in dimension $z$. In particular $\Kfam(n,\bw)=\emptyset$ when $z=1$.
\end{theorem}

\begin{proof}
Replacing $t$ by $t^{-1}$ we may take the nonzero weights positive. Write the source coordinates as $(x,\zeta)$, with $x=(x_1,\dots,x_k)$ carrying the positive weights, $k=n-z$, and $\zeta=(\zeta_1,\dots,\zeta_z)$ the weight-zero ones, and the target coordinates likewise. A component of weight $0$ is an invariant, hence lies in $\C[\zeta]$; a component of positive weight has every monomial divisible by some $x_i$. So
\[
  G(x,\zeta)=\bigl(F(x,\zeta),\,h(\zeta)\bigr),
\]
and $\Jm{G}$ is block triangular with $\detJ{G}=\det(\partial F/\partial x)\cdot\det(\partial h/\partial\zeta)$. A product of two polynomials equal to a nonzero constant has both factors nonzero constants, so $h$ is a Keller map in $z$ variables and, for each fixed $\zeta_0$, the map $F(\cdot,\zeta_0)\colon\A^k\to\A^k$ is a Keller map, graded with elliptic weights $\bw'$. By \cref{thm:rigid} it is an automorphism of $\A^k$.

Hence $G$ is injective if and only if $h$ is: one direction is immediate, and if $h(\zeta)=h(\zeta')$ with $\zeta\neq\zeta'$ then, $F(\cdot,\zeta')$ being surjective, any $x$ admits an $x'$ with $F(x,\zeta)=F(x',\zeta')$. An injective polynomial self-map of $\C^n$ is an automorphism \cite{rudin}, so $G$ is an automorphism exactly when $h$ is. Conversely, a Keller map $h$ of $\A^z$ which is not an automorphism yields $G=(x,h(\zeta))$, graded for $\bw=(1,\dots,1,0,\dots,0)$ and not an automorphism. For $z=1$ a Keller map of $\A^1$ is affine.
\end{proof}

Suppose instead that $\bw'$ is hyperbolic. If $\bw$ has exactly one positive weight and that weight is $1$, so that
\[
  \bw=(1,-r_2,\dots,-r_n), \qquad r_i\ge0,
\]
then the whole of \cref{sec:higher} applies with $r_i\ge0$ in place of $r_i>0$. The invariants $u_i=x^{r_i}y_i$ still generate $\C[x,y]^{\sigma}$ freely, the weight-$1$ piece is still $x\,\C[u]$, and $V_i=\C[u]$ when $r_i=0$; \cref{prop:ndescent} and \cref{thm:nmaster} are unchanged, the summands of $\mu$ indexed by $r_i=0$ being zero. Only \cref{lem:nbasept} needs amending: $B_i(O)=0$ is available for those $i$ with $r_i>0$, which is all that its proof uses, and the conclusions $\Lambda(O)\neq0$ and $\Jac_{u}(B_2,\dots,B_n)(O)=\kappa/\Lambda(O)$ persist.

\begin{prop}\label{prop:parabolic-hyp}
Let $\bw=(1,-r_2,\dots,-r_n)$ with $r_i\ge0$ and $m=\sum_{i\ge2}r_i$. If $m\ge2$, then \cref{thm:nnoninj} holds verbatim: a solution of \cref{eq:nmaster} with $\Lambda\notin\C$ is a counterexample to the Jacobian Conjecture, and the family is parametrised by the scheme $\Yparo[\br,\bd]$ of \cref{sec:higher}. If $m=1$, that is if $\bw=(1,-1,0,\dots,0)$, the exponent $m-1$ vanishes, $\detJ{G}=\pm\Jac_{u}(P_2,\dots,P_n)$, and the Keller condition for $G$ is exactly the Keller condition for $\Gbar$ in dimension $n-1$, with no contracted locus, as in \cref{rem:vasyunin}.
\end{prop}

So the parabolic case is not a separate difficulty: with an elliptic nonzero part it is the Jacobian Conjecture in dimension $z$, and with a hyperbolic nonzero part of the above type it is the theory of \cref{sec:higher}. What is left open is the parabolic case whose hyperbolic part falls outside that type, and that is the same restriction we meet next.

\subsection{Third cut: the shape of the hyperbolic quotient}
\label{subsec:shape}

Let $\bw$ be hyperbolic, and let $n_+$ be the number of positive weights. Replacing $\bw$ by $-\bw$ if necessary, assume $n_+\le n-n_+$. For $n=3$ this forces $n_+=1$: the case of two positive weights, such as $\bw=(1,1,-1)$, is carried by the substitution $t\mapsto t^{-1}$ and a coordinate permutation to $\bw=(1,-1,-1)$, which is the type \cref{eq:goodtype} with $(r,s)=(1,1)$ and is treated below. The first genuinely new configuration with $n_+\ge2$ therefore occurs in dimension four.

The construction of \cref{sec-8} and \cref{sec:higher} requires $n_+=1$ and the positive weight equal to $1$. The two hypotheses fail for different reasons.

If $n_+=1$ with positive weight $w_1=a\ge2$, the invariant ring is generated by the monomials $x^{c}y^{\bb}$ with $ca=\sum_i r_ib_i$, where $\bw=(a,-r_2,\dots,-r_n)$, and this semigroup need not be free. Already $\bw=(2,-1,-1)$ gives $\C[xy^2,xyz,xz^2]$, a quadric cone, as noted in \cref{rem:exponent}: the reduction then takes place on a singular quotient, where $\Jac_{u}$ has no elementary meaning and the graded pieces are modules over a non-polynomial ring rather than monomial ideals.

If $n_+\ge2$, so $n\ge4$, the weight-$1$ piece is no longer generated by a single element over the invariants. For $\bw=(1,1,-r_3,-r_4)$ it is $x_1R+x_2R$ with $R$ the invariant ring, since a monomial of weight $1$ is divisible by $x_1$ or by $x_2$ but need not be divisible by a prescribed one. The description \cref{eq:ntuple}, and with it the single factor $\Lambda$ and the contracted locus $\{\Lambda=0\}$ around which \cref{sec-8} is organised, has no analogue. The quotient is in general singular here as well: for $\bw=(1,1,-1,-1)$ the invariant ring is the affine cone over $\bP^1\times\bP^1$.

We therefore restrict to
\begin{equation}\label{eq:goodtype}
  \bw=(1,-r_2,\dots,-r_n), \qquad r_i>0, \qquad \bq=\bw,
\end{equation}
the setting of \cref{sec:higher}, and record the two omitted configurations as open. Within the hyperbolic case the classification is thus governed by more than the signature: by the number of positive weights, and by the value of the positive one.

\subsection{The two ends of the filtration}
\label{subsec:filtration}

Fix a type \cref{eq:goodtype}. Two constraints on a point of $\Yparo[\br,\bd]$ are visible without any computation, one at the top of the degree filtration and one at the bottom. We record both, since the emptiness proofs below use them together.

For the top, fix $\bd=(d_2,\dots,d_n,d_{\Lambda})$ and write $\Yeq[\br,\bd]\subseteq\Yparo[\br,\bd]$ for the locally closed subscheme of points whose data has these degrees \emph{exactly}, so that
\[
  \Yparo[\br,\bd]=\bigsqcup_{\bd'\le\bd,\ d'_{\Lambda}\ge1}\Yeq[\br,\bd'] .
\]
Write $\hat B_i$ and $\hat\Lambda$ for the leading homogeneous forms of $B_i$ and $\Lambda$ in the standard grading of $\C[u]$, and recall
\[
  D=d_{\Lambda}+\sum_{i\ge2}d_i-(n-1).
\]

\begin{theorem}\label{thm:leading}
Let $[G]\in\Yeq[\br,\bd]$. Then $\mu$ has degree at most $D$, and its degree-$D$ component is $\mu(\hat B_2,\dots,\hat B_n,\hat\Lambda)$. Since $D\ge1$, the Keller condition \cref{eq:nmaster} forces
\[
  \mu(\hat B_2,\dots,\hat B_n,\hat\Lambda)=0,
  \qquad\text{equivalently}\qquad
  \Jac_{u}\bigl(\hat B_2\hat\Lambda^{r_2},\dots,\hat B_n\hat\Lambda^{r_n}\bigr)=0 ,
\]
that is, the $n-1$ forms $\hat B_i\hat\Lambda^{r_i}$ are algebraically dependent over $\C$.
\end{theorem}

\begin{proof}
The summand $\Lambda\Jac_{u}(B_2,\dots,B_n)$ has degree at most $d_{\Lambda}+\sum_i(d_i-1)=D$, and the summand $r_iB_i\Jac_{u}(B_2,\dots,\Lambda,\dots,B_n)$ has degree at most
\[
  d_i+(d_{\Lambda}-1)+\sum_{j\neq i}(d_j-1)=D ,
\]
and in each the degree-$D$ part is obtained by replacing every argument by its leading form. Hence the degree-$D$ component of $\mu$ is $\mu(\hat B_2,\dots,\hat B_n,\hat\Lambda)$. Since $\deg\Lambda\ge1$ and each $B_i$ lies in the maximal ideal, $D\ge1$, while $\mu=\kappa$ has degree $0$; so that component vanishes. Applying \cref{thm:nmaster} to the forms gives
\[
  \mu(\hat B_2,\dots,\hat B_n,\hat\Lambda)
  =\hat\Lambda^{-(m-1)}\Jac_{u}\bigl(\hat B_2\hat\Lambda^{r_2},\dots,\hat B_n\hat\Lambda^{r_n}\bigr),
\]
and the last assertion is the Jacobian criterion for algebraic dependence in characteristic zero.
\end{proof}

For the bottom of the filtration, write $\check B_i$ for the initial form of $B_i$ at $O$.

\begin{lem}\label{lem:initial}
Let $[G]\in\Yparo[\br,\bd]$. Then each $B_i$ has order exactly $1$ at $O$, the linear forms $\check B_2,\dots,\check B_n$ are linearly independent, and
\[
  \Jac_{u}(\check B_2,\dots,\check B_n)=\kappa/\Lambda(O)\in\C^{\times}.
\]
For $n=3$, if moreover $r<s$ then $\check A\in\C v$ and $\check B$ has a nonzero $u$-coefficient.
\end{lem}

\begin{proof}
By \cref{lem:nbasept}, $B_i(O)=0$ and $\Lambda(O)\Jac_{u}(B_2,\dots,B_n)(O)=\kappa\neq0$. The Jacobian at $O$ is the determinant of the matrix of linear parts, so those linear parts are independent and in particular nonzero, which is the statement about orders. For $n=3$ the degree-one part of $\VA=\Span\{u^iv^j:ri+sj\ge s\}$ is $\C v$ when $r<s$, so $\check A=a_1v$ with $a_1\neq0$, and $\jb{\check B}{\check A}=a_1\partial_u\check B\neq0$.
\end{proof}

Note that \cref{lem:initial} leaves no room for a lower-order analogue of \cref{thm:leading}: the initial forms are never degenerate.

\subsection{Dimension three: the leading system}
\label{subsec:leadingsystem}

For $n=3$ the forms are binary and \cref{thm:leading} becomes explicit. Write $(r_2,r_3)=(r,s)$, $B_2=B$, $B_3=A$ and $\bd=(d_1,d_2,d_3)$ as in \cref{sec-8}, and set
\[
  \nu_A=d_1+sd_3=\deg\bigl(\hat A\hat\Lambda^{s}\bigr), \qquad
  \nu_B=d_2+rd_3=\deg\bigl(\hat B\hat\Lambda^{r}\bigr), \qquad
  e=\gcd(\nu_A,\nu_B).
\]

\begin{theorem}\label{thm:binary}
Let $[G]\in\Yeq$, and let $\hat A,\hat B,\hat\Lambda$ be the leading forms, of degrees $d_1,d_2,d_3$. Then there are a binary form $h$ of degree $e$ and constants $\kappa_1,\kappa_2\in\C^{\times}$ with
\[
  \hat B\hat\Lambda^{r}=\kappa_1h^{\nu_B/e}, \qquad
  \hat A\hat\Lambda^{s}=\kappa_2h^{\nu_A/e} .
\]
Consequently $\hat A$, $\hat B$ and $\hat\Lambda$ are, up to constants, products of powers of the same linear forms: writing $h=\prod_{i=1}^{k}\ell_i^{m_i}$ with the $\ell_i$ pairwise non-proportional and $m_i\ge1$, one has $\hat\Lambda=\prod\ell_i^{\gamma_i}$, $\hat B=\prod\ell_i^{\beta_i}$, $\hat A=\prod\ell_i^{\alpha_i}$, where the nonnegative integers $\alpha_i,\beta_i,\gamma_i,m_i$ satisfy
\begin{equation}\label{eq:diophantine}
  \beta_i+r\gamma_i=\tfrac{\nu_B}{e}\,m_i, \qquad
  \alpha_i+s\gamma_i=\tfrac{\nu_A}{e}\,m_i,
\end{equation}
together with $\sum_i\alpha_i=d_1$, $\sum_i\beta_i=d_2$, $\sum_i\gamma_i=d_3$, $\sum_i m_i=e$.
\end{theorem}

\begin{proof}
Put $f=\hat B\hat\Lambda^{r}$ and $g=\hat A\hat\Lambda^{s}$, nonzero forms of degrees $\nu_B$ and $\nu_A$, satisfying $\jb{f}{g}=0$ by \cref{thm:leading}. Euler's relation gives $\nu_Bf=uf_u+vf_v$ and $\nu_Ag=ug_u+vg_v$; combining with $f_ug_v=f_vg_u$ yields $\nu_Bfg_u=\nu_Af_ug$ and $\nu_Bfg_v=\nu_Af_vg$, so the rational function $f^{\nu_A}/g^{\nu_B}$, of degree zero, has vanishing partial derivatives and is constant. Hence $f^{\nu_A}=cg^{\nu_B}$ for some $c\in\C^{\times}$.

Over $\C$ both forms split. Let $\ell_1,\dots,\ell_k$ be the distinct linear factors occurring in $f$ or in $g$, with multiplicities $\varphi_i$ in $f$ and $\gamma'_i$ in $g$. Comparing multiplicities, $\nu_A\varphi_i=\nu_B\gamma'_i$ for every $i$. Writing $\nu_A=e\nu_A'$ and $\nu_B=e\nu_B'$ with $\gcd(\nu_A',\nu_B')=1$, this reads $\nu_A'\varphi_i=\nu_B'\gamma'_i$, so $\nu_B'\mid\varphi_i$; set $m_i=\varphi_i/\nu_B'$, whence $\gamma'_i=\nu_A'm_i$ and $\sum_im_i=e$. With $h=\prod\ell_i^{m_i}$ we get $f=\kappa_1h^{\nu_B'}$ and $g=\kappa_2h^{\nu_A'}$. The system \cref{eq:diophantine} is the equality of multiplicities at $\ell_i$ in these two identities, and the four sums record the degrees.
\end{proof}

\begin{cor}\label{cor:empty}
If $e=1$, then every point of $\Yeq$ has $k=1$: the three leading forms are powers of a single linear form.
\end{cor}

\begin{proof}
$k\le\sum_im_i=e$.
\end{proof}

\begin{rem}\label{rem:notempty}
\cref{thm:binary} is a stratification of the leading data, not by itself an emptiness criterion. Indeed \cref{eq:diophantine} always admits the solution $k=1$, $m=(e)$, $\gamma=(d_3)$, for which $\beta_1=\nu_B-rd_3=d_2$ and $\alpha_1=\nu_A-sd_3=d_1$ and the inequalities of \cref{eq:search} reduce to $d_1,d_2\ge0$. Emptiness must therefore be proved stratum by stratum, by combining \cref{thm:binary} with \cref{lem:initial} or with the linear algebra of \cref{prop:determinantal}. Both routes are taken below.
\end{rem}

\subsection{Three applications}
\label{subsec:applications}

\begin{exa}\label{ex:Fleading}
For \cref{eq:map} one has $r=1$, $s=2$, $\bd=(4,3,1)$, hence $\nu_B=4$, $\nu_A=6$, $e=2$. From $A=(1+u)M$, $B=u+3M$ and $\Lambda=L$ one computes the leading forms
\[
  \hat A=u^{3}(3u+v), \qquad \hat B=3u^{2}(3u+v), \qquad \hat\Lambda=-(3u+v),
\]
so that $\hat B\hat\Lambda=-3u^{2}(3u+v)^{2}$ and $\hat A\hat\Lambda^{2}=u^{3}(3u+v)^{3}$, and indeed
\[
  (\hat B\hat\Lambda)^{6}=729\,(\hat A\hat\Lambda^{2})^{4}.
\]
Here $h=u(3u+v)$, with $\ell_1=u$, $\ell_2=3u+v$ and $m=(1,1)$, and \cref{eq:diophantine} reads
\[
  \beta=(2,1), \qquad \gamma=(0,1), \qquad \alpha=(3,1),
\]
which satisfies $\beta_i+\gamma_i=2m_i$ and $\alpha_i+2\gamma_i=3m_i$ for $i=1,2$. The leading data of the announced counterexample is thus the solution of a small system of linear equations in nonnegative integers, and not an accident.
\end{exa}

The next result is the first place where the two ends of the filtration are used together. It supersedes the computation of \cref{rem:weightoperator} below, which treated $(r,s)=(1,1)$ only.

\begin{theorem}\label{prop:oneone}
Let $r,s>0$ and $d_3\ge1$. Then $\Yparo[r,s,(1,1,d_3)]=\emptyset$. Equivalently, no graded Keller counterexample of type \cref{eq:goodtype} in dimension three has both $A$ and $B$ of degree one in the invariants.
\end{theorem}

\begin{proof}
Since $A\in\VA$ and $B\in\VB$ contain no constants, a point of $\Yparo[r,s,(1,1,d_3)]$ lies in $\Yeq[r,s,(1,1,d_3')]$ for some $1\le d_3'\le d_3$; so we may assume $\deg A=\deg B=1$ exactly. Then $\hat A=\check A$ and $\hat B=\check B$, and \cref{lem:initial} gives $\jb{\hat B}{\hat A}\neq0$, so $\hat A$ and $\hat B$ are non-proportional linear forms.

By \cref{thm:binary} we have $\hat A=\prod\ell_i^{\alpha_i}$ and $\hat B=\prod\ell_i^{\beta_i}$ with $\sum\alpha_i=\sum\beta_i=1$. So there are indices $i$ and $j$ with $\alpha_i=1$, $\beta_j=1$ and all other $\alpha$'s and $\beta$'s zero, and non-proportionality forces $i\neq j$. Substituting into \cref{eq:diophantine},
\[
  \tfrac{\nu_B}{e}m_i=r\gamma_i, \qquad
  \tfrac{\nu_A}{e}m_i=s\gamma_i+1, \qquad
  \tfrac{\nu_B}{e}m_j=r\gamma_j+1, \qquad
  \tfrac{\nu_A}{e}m_j=s\gamma_j .
\]
Eliminating $\gamma_i$ from the first pair gives $m_i(r\nu_A-s\nu_B)=re$, and eliminating $\gamma_j$ from the second pair gives $m_j(s\nu_B-r\nu_A)=se$. Since $d_1=d_2=1$ we have $r\nu_A-s\nu_B=rd_1-sd_2=r-s$, so
\[
  m_i(r-s)=re>0 \qquad\text{and}\qquad m_j(s-r)=se>0 ,
\]
using $m_i,m_j\ge1$ and $e\ge1$. These require $r>s$ and $s>r$ at once, a contradiction.
\end{proof}

\begin{rem}\label{rem:weightoperator}
For $(r,s)=(1,1)$ the conclusion can be seen directly, and the computation exhibits an operator worth recording. Write $A=a_1u+a_2v$, $B=b_1u+b_2v$ and $\delta=b_1a_2-b_2a_1$. Then $\jb{B}{A}=\delta$ and $Ab_1-Ba_1=\delta v$, $Ba_2-Ab_2=\delta u$, so \cref{eq:master} with $r=s=1$ collapses to
\[
  \mu=\delta\bigl(\Lambda+u\Lambda_u+v\Lambda_v\bigr)
     =\delta\sum_{i,j}(1+i+j)\Lambda_{ij}u^iv^j .
\]
If $\delta=0$ then $\mu=0$, excluded; if $\delta\neq0$ then $\mu$ is constant only if $\deg\Lambda=0$, again excluded. The operator appearing here is the weight operator of \cref{prop:wps}, whose eigenvalue on the coefficient of $u^iv^j$ is the weight $1+i+j$ of the weighted projective space $\WP(\bomega)$. We do not know whether this is more than a coincidence.
\end{rem}

By \cref{prop:oneone} the smallest cases left open in dimension three are $(r,s)=(1,1)$ with $\bd=(2,1,d_3)$ and $\bd=(1,2,d_3)$; whether $\Kfam(3,(1,-1,-1))$ is empty we do not know.

\subsection{Constructing counterexamples}
\label{subsec:search}

\cref{thm:binary} turns the construction of a graded counterexample of given exact degrees into three steps, the first finite and the third linear. We describe them, and then carry them out at the degrees of \cref{eq:map}.

\medskip
\noindent\textbf{Step 1: the integer enumeration.}
Fix $(r,s)$ and $\bd=(d_1,d_2,d_3)$, and let $\nu_A,\nu_B,e$ be as above. Summing the two families in \cref{eq:diophantine} over $i$ gives $\sum_i\beta_i=\nu_B-rd_3=d_2$ and $\sum_i\alpha_i=\nu_A-sd_3=d_1$, so the degree conditions on $\alpha$ and $\beta$ hold automatically. The data to be enumerated is only the pair $(m,\gamma)$, and \cref{eq:diophantine} reduces to
\begin{equation}\label{eq:search}
  \sum_i m_i=e, \qquad \sum_i \gamma_i=d_3, \qquad
  \tfrac{\nu_B}{e}m_i\ge r\gamma_i, \qquad \tfrac{\nu_A}{e}m_i\ge s\gamma_i ,
\end{equation}
in nonnegative integers, with $\alpha_i=\tfrac{\nu_A}{e}m_i-s\gamma_i$ and $\beta_i=\tfrac{\nu_B}{e}m_i-r\gamma_i$. If $m_i=0$ then $\beta_i=-r\gamma_i\ge0$ forces $\gamma_i=0$ and hence $\alpha_i=\beta_i=0$, so $\ell_i$ does not occur; we may assume $m_i\ge1$, whence $1\le k\le e$. The enumeration is over compositions of $e$ into $k$ positive parts and of $d_3$ into $k$ nonnegative parts, a finite list, and it produces a finite set of strata covering $\Yeq$. By \cref{rem:notempty} the list is never empty, so this step does not by itself decide emptiness; what it does is replace one search by finitely many much smaller ones.

\medskip
\noindent\textbf{Step 2: normalising the leading forms.}
A solution of \cref{eq:search} determines $\hat A,\hat B,\hat\Lambda$ up to the choice of the $k$ pairwise non-proportional linear forms $\ell_i$ and of scalars. The scalars are absorbed by the $\Gm^{3}$-action of \cref{thm:points}. The linear forms are acted on by the linear part of $\Gaut$, which depends on $(r,s)$: the degree-one part of $\VA$ is $\C v$ when $r<s$ and $\C u+\C v$ when $r\ge s$, and that of $\VB$ is $\C u$ when $r>s$ and $\C u+\C v$ when $r\le s$. Hence the linear part of $\Gaut$ is
\[
  \GL_2 \ \text{ if } r=s, \qquad \text{a Borel subgroup of } \GL_2 \ \text{ if } r\neq s .
\]
Modulo scalars this is $\mathrm{PGL}_2$, acting with finitely many orbits on configurations of $k\le3$ points of $\bP^1$, respectively the affine group $\A^1\rtimes\Gm$, acting with finitely many orbits on configurations of $k\le2$ points. So the leading data carries no continuous modulus when $k\le 3$ (case $r=s$) or $k\le2$ (case $r\neq s$).

\medskip
\noindent\textbf{Step 3: the lower-order terms.}
With the leading forms fixed, what remains is to solve \cref{eq:master} for the lower-order coefficients of $A$, $B$ and $\Lambda$. This is the situation of \cref{prop:determinantal}: fixing $B$ and $\Lambda$, the equation is linear in $A$, and a solution with $\kappa\neq0$ exists exactly when
\[
  \rank\TBL=\rank(\pi\circ\TBL)+1 .
\]
The search space is much smaller than the ambient $\bP(\VBd)\times\bP(\VLd)$ of \cref{sec-8}, since the leading coefficients of $B$ and $\Lambda$ are prescribed by Step~1 and normalised by Step~2.

\medskip
The scheme $\Yparo$ is defined over $\Z$, so Steps 1--3 may be carried out over $\F_p$. Here some care is needed: a $\Q$-point of $\Yparo$ spreads out to a $\Z$-section of $\Ypar$, but the open conditions $\mu_{00}\neq0$ and $\deg\Lambda\ge1$ may fail on reduction, and the primes at which they fail depend on the point. So $\Yparo(\F_p)=\emptyset$ for a single $p$ does not prove $\Yparo(\Q)=\emptyset$; what it proves is that every $\Q$-point degenerates at $p$. Emptiness over $\Q$ follows only from emptiness over $\F_p$ for all $p$ outside an effectively bounded set, or from emptiness over $\overline{\Q}$, for which the elimination over algebraically closed fields recorded in \cref{sec-8} applies.

\begin{exa}\label{ex:enumerate}
Take $(r,s)=(1,2)$ and $\bd=(4,3,1)$, the degrees of \cref{eq:map}. Then $\nu_B=4$, $\nu_A=6$, $e=2$, so $\nu_B/e=2$ and $\nu_A/e=3$, and \cref{eq:search} reads
\[
  \sum_i m_i=2, \qquad \sum_i \gamma_i=1, \qquad 2m_i\ge \gamma_i, \qquad 3m_i\ge 2\gamma_i .
\]
Since $k\le e=2$, there are exactly two solutions up to relabelling:
\[
  \begin{array}{c|cccc}
    k & m & \gamma & \beta & \alpha \\\hline
    1 & (2) & (1) & (3) & (4) \\
    2 & (1,1) & (0,1) & (2,1) & (3,1)
  \end{array}
\]
The second is realised by \cref{eq:map}, as computed in \cref{ex:Fleading}, with $\ell_1=u$ and $\ell_2=3u+v$. The first has all three leading forms equal to powers of a single linear form $\ell$,
\[
  \hat\Lambda=\ell, \qquad \hat B=\ell^{3}, \qquad \hat A=\ell^{4},
\]
up to constants. In both solutions $k\le2$, so by Step~2 the leading data carries no continuous modulus. This is consistent with, but weaker than, the possibility raised in \cref{rem:orbit} that $\Yparo[1,2,(4,3,1)]$ is a single $\Gaut$-orbit: by \cref{prop:determinantal} the fibres of $\Ypar\to\bP(\VBd)\times\bP(\VLd)$ are linear spaces, so a family may well move while its leading data stays fixed. What Step~2 does exclude is a family along which the leading data itself varies; by $k\le e$, such a family requires $e\ge3$ when $r\neq s$, and $e\ge4$ when $r=s$.
\end{exa}

\begin{prop}\label{prop:konestratum}
Let $(r,s)=(1,2)$ and $\bd=(4,3,1)$. Then the stratum $k=1$ of $\Yeq[1,2,(4,3,1)]$ is empty. By \cref{ex:enumerate} every point of $\Yeq[1,2,(4,3,1)]$ therefore has $k=2$, the leading configuration of \cref{eq:map}.
\end{prop}

\begin{proof}
Suppose $[G]\in\Yeq[1,2,(4,3,1)]$ has $k=1$, so that by \cref{thm:binary} the leading forms of $A,B,\Lambda$ are, up to constants, $\ell^{4}$, $\ell^{3}$ and $\ell$ for a single linear form $\ell$.

\emph{Normalisation.} By Step~2 the linear part of $\Gaut$ consists of the invertible $\varphi$ with $\varphi^{*}u=a_{11}u+a_{12}v$ and $\varphi^{*}v=a_{22}v$, since $r<s$. Such a $\varphi$ preserves all three degrees and carries $\ell$ to $\varphi^{*}\ell$. Writing $\ell=c_uu+c_vv$, if $c_u\neq0$ the choice $a_{12}=-c_va_{22}/c_u$ gives $\varphi^{*}\ell\propto u$, while if $c_u=0$ then $\varphi^{*}\ell\propto v$ for every $\varphi$. So we may assume $\ell=u$ or $\ell=v$. In the first case $\Lambda=\lambda_{10}u+\lambda_{00}$ with $\lambda_{10}\neq0$ and $\lambda_{00}=\Lambda(O)\neq0$ by \cref{lem:basept}; applying the diagonal element $u\mapsto\mu u$, $v\mapsto\nu v$ of $\Gaut$ and the scaling $\Lambda\mapsto\gamma\Lambda$ of \cref{thm:points} with $\gamma=\lambda_{00}^{-1}$ and $\mu=\lambda_{00}/\lambda_{10}$, we may take $\Lambda=U:=u+1$. The second case is normalised in the same way to $\Lambda=V:=v+1$.

\emph{Standing constraints.} By \cref{lem:initial}, applied with $r=1<2=s$,
\begin{equation}\label{eq:kone-nonvanish}
  A_u(O)=0, \qquad A_v(O)\neq0, \qquad B(O)=0, \qquad B_u(O)\neq0 .
\end{equation}
Moreover $A\in\VA=(u^{2},v)$ gives $A(u,0)\in(u^{2})$, and $B\in\VB=(u,v)$ gives $B(O)=0$.

\emph{Common skeleton.} In both cases we expand \cref{eq:master} in powers of the variable transverse to $\ell$, extract the top coefficient, integrate the resulting logarithmic identity, and compare degrees. The leading forms bound the expansion: from $\hat A=\ell^4$ and $\hat B=\ell^3$, no monomial of top degree involves the transverse variable, so if $A=\sum_j\alpha_jt^j$ and $B=\sum_j\beta_jt^j$ in the transverse variable $t$, then
\begin{equation}\label{eq:kone-ode}
  \deg\alpha_j\le3-j\ (j\ge1), \quad \deg\alpha_0=4, \qquad
  \deg\beta_j\le2-j\ (j\ge1), \quad \deg\beta_0=3 ,
\end{equation}
whence $N:=\deg_tA\le3$ and $M:=\deg_tB\le2$.

\emph{Case $\ell=u$.} Here $\Lambda_u=1$, $\Lambda_v=0$, so $\jb{B}{\Lambda}=-B_v$ and $\jb{\Lambda}{A}=A_v$, and \cref{eq:master} reads
\begin{equation}\label{eq:kone-U}
  U\bigl(B_uA_v-B_vA_u\bigr)-2AB_v+BA_v=\kappa .
\end{equation}
Take $t=v$. By \cref{eq:kone-nonvanish}, $A_v(O)=\alpha_1(0)\neq0$, so $N\ge1$. If $M=0$ then $B_v=0$ and \cref{eq:kone-U} becomes $A_v(U\beta_0)'=\kappa$ with $'=d/du$, so both factors are constants; then $U\beta_0$ is linear with zero constant term, so $\beta_0$ is a constant, and $\beta_0(0)=B(O)=0$ gives $\kappa=0$. So $M\ge1$.

Set $f=\beta_M$, $g=\alpha_N$. Every term of \cref{eq:kone-U} has $v$-degree at most $M+N-1\ge1$, so the coefficient of $v^{M+N-1}$ vanishes:
\[
  U\bigl(Nf'g-Mfg'\bigr)+(N-2M)fg=0,
  \qquad\text{hence}\qquad
  f^{\,N}=C\,U^{\,2M-N}g^{\,M},\quad C\in\C^{\times},
\]
and $N\deg f-M\deg g=2M-N$. With $1\le M\le2$, $1\le N\le3$ and the bounds \cref{eq:kone-ode}, the only solutions are $(M,N)=(1,1)$ with $(\deg f,\deg g)=(1,0)$ and $(M,N)=(1,2)$ with $(\deg f,\deg g)=(0,0)$.

If $(M,N)=(1,1)$ then $\alpha_1=a\in\C^{\times}$ and $\beta_1=bU$ with $b\in\C^{\times}$, so $A=\alpha_0+av$ and $B=\beta_0+bUv$. Substituting into \cref{eq:kone-U}, the terms in $v$ cancel identically and there remains $a(U\beta_0)'-b(U^{2}\alpha_0)'=\kappa$. Integrating gives $aU\beta_0-bU^{2}\alpha_0=\kappa U+c_1$, and $c_1=0$ by evaluation at $U=0$; dividing by $U$ gives $a\beta_0-bU\alpha_0=\kappa$. At $O$ one has $U=1$, $\beta_0(0)=B(O)=0$ and $\alpha_0(0)=0$ since $u^2\mid\alpha_0$, so $\kappa=0$, a contradiction.

If $(M,N)=(1,2)$ then $\beta_1=b$ and $\alpha_2=g$ are nonzero constants. The coefficient of $v^{2}$ in \cref{eq:kone-U} vanishes identically and that of $v$ gives $2g(U\beta_0)'-b(U\alpha_1)'=0$; integrating and evaluating at $U=0$ gives $2g\beta_0=b\alpha_1$, impossible since $\deg\beta_0=3$ and $\deg\alpha_1\le2$.

\emph{Case $\ell=v$.} Here $\Lambda_u=0$, $\Lambda_v=1$, so $\jb{B}{\Lambda}=B_u$ and $\jb{\Lambda}{A}=-A_u$, and \cref{eq:master} reads
\begin{equation}\label{eq:kone-V}
  V\bigl(B_uA_v-B_vA_u\bigr)+2AB_u-BA_u=\kappa .
\end{equation}
Take $t=u$. By \cref{eq:kone-nonvanish}, $B_u(O)=\beta_1(0)\neq0$, so $M\ge1$. If $N=0$ then $A_u=0$ and \cref{eq:kone-V} becomes $B_u(V\alpha_0'+2\alpha_0)=\kappa$ with $'=d/dv$, so both factors are constants; writing $\alpha_0=\sum_kc_kV^{k}$ gives $V\alpha_0'+2\alpha_0=\sum_k(k+2)c_kV^{k}$, constant only if $\deg\alpha_0=0$, against \cref{eq:kone-ode}. So $N\ge1$.

With $f=\beta_M$ and $g=\alpha_N$, the coefficient of $u^{M+N-1}$ in \cref{eq:kone-V} vanishes, giving
\[
  V\bigl(Mfg'-Nf'g\bigr)+(2M-N)fg=0,
  \qquad\text{hence}\qquad
  g^{\,M}=C\,V^{\,N-2M}f^{\,N},
\]
and $M\deg g-N\deg f=N-2M$. Since $A\in\VA$, the coefficient of $u$ in $A$ is divisible by $v$, so $\alpha_1(0)=0$ and $\deg\alpha_1\ge1$ whenever $\alpha_1\neq0$. Enumerating with \cref{eq:kone-ode}: $(M,N)=(1,1)$ forces $\deg g=\deg\alpha_1=0$, contradicting this; $(M,N)=(1,3)$, $(2,1)$, $(2,2)$ and $(2,3)$ are excluded by the degree relation; and $(M,N)=(1,2)$ forces $\deg f=\deg g=0$.

In that case $\beta_1=b$ and $\alpha_2=g$ are nonzero constants, so $A=\alpha_0+\alpha_1u+gu^{2}$ and $B=\beta_0+bu$. The coefficient of $u^{2}$ in \cref{eq:kone-V} vanishes identically, and that of $u$ gives $b(V\alpha_1)'=2g(V\beta_0)'$; integrating and evaluating at $V=0$ gives $b\alpha_1=2g\beta_0$, impossible since $\deg\beta_0=3$ and $\deg\alpha_1\le2$.

Both normalised cases are impossible, so the stratum $k=1$ is empty.
\end{proof}

\begin{rem}\label{rem:substrata}
\cref{prop:konestratum} concerns the exact-degree stratum, which is where $[F]$ lives. The bounded scheme $\Yparo[1,2,(4,3,1)]$ also contains the strata $\Yeq[1,2,(d_1',d_2',1)]$ with $d_1'\le4$, $d_2'\le3$; of these, $(d_1',d_2')=(1,1)$ is empty by \cref{prop:oneone}, and the remaining ten are not settled here. The two contradictions in the proof above both read ``$\deg\beta_0=3$ while $\deg\alpha_1\le2$'', so they use the exact degrees essentially and do not transfer.
\end{rem}

\subsection{Summary}  \label{subsec:summary}

For the fourth row the family is the set of rational points of the explicit scheme $\Yparo[\br,\bd]$ of \cref{sec:higher}; its leading stratum is constrained by \cref{thm:leading}, in dimension three by the system \cref{eq:diophantine}, and its initial stratum by \cref{lem:initial}. The map \cref{eq:map} is one of its points, and whether it is the only one, up to $\Gaut$, is \cref{rem:orbit}. Three facts frame that question at the degrees of \cref{eq:map}: the leading configuration is forced to be that of \cref{eq:map} (\cref{prop:konestratum}), the leading data carries no continuous modulus (\cref{ex:enumerate}), and a family with varying leading data would require $e\ge3$. None of these excludes a family with constant leading data, which is what \cref{prop:determinantal} would have to decide.

\begin{table}[h]
\[
\begin{array}{l|l}
 \text{type of } \bw \ (\bw\neq0,\ \gcd=1) & \text{status of } \Kfam(n,\bw) \\\hline
 \text{elliptic, any } n & \emptyset \ (\text{\cref{thm:firstcut}}) \\
 n\le2, \text{ any signature} & \emptyset \ (\text{\cref{thm:firstcut}}) \\
 \text{parabolic, nonzero part elliptic, } z \text{ zeros} & \neq\emptyset \iff \text{JC fails in dim } z;\\
 					&	 \emptyset \text{ for } z=1 \\
 \bw=(1,-r_2,\dots,-r_n),\ r_i\ge0,\ m=\sum r_i\ge2 & \text{the rational points of } \Yparo[\br,\bd] \\
 \bw=(1,-1,0,\dots,0) & \text{JC in dimension } n-1, \\
 				&	\text{ no contracted locus} \\
 \text{hyperbolic, one positive weight } \ge2 & \text{open (singular quotient)} \\
 \text{hyperbolic, } n_+\ge2 \ (\text{so } n\ge4) & \text{open (weight-1 piece not cyclic)}
\end{array}
\]
\caption{The classification of \cref{sec:classify}.}
\label{tab:classify}
\end{table}

Finally, the parameter space itself is of the opposite signature to the map it parametrises. By \cref{prop:wps} the triples $(A,B,\Lambda)$ modulo the natural $\Gm$ form a weighted projective space $\WP(\bomega)$ with all weights positive, so the arithmetic of a hyperbolic counterexample is governed by an elliptic moduli space, of exactly the type studied in \cite{sparsity}. This is the sharpest form of the comparison announced in the introduction, and it is the reason the counting function of \cref{prop:wps} is well posed.


\section{Concluding remarks}\label{sec:9}

The results above leave a picture with some proved edges and some open middle. We close by recording both: why the mechanism of \cref{eq:map} needs dimension three, where the classification of \cref{sec:classify} actually stops, what a counting theory over $\Q$ would require, and how the two papers' settings fit together.

\subsection*{Why dimension three}

For a weight type \cref{eq:goodtype} the Keller condition is, by \cref{prop:ndescent}, the single divisorial identity
\begin{equation}\label{eq:concl-master}
  \Jac_{u}(P_2,\dots,P_n)=\kappa\,\Lambda^{\,m-1}, \qquad m=\sum_{i\ge2}r_i:
\end{equation}
the Jacobian of the quotient map vanishes along the contracted locus $\{\Lambda=0\}$, to the order prescribed by the weights, and nowhere else. In dimension two this is an identity between polynomials in one variable, and there $\Jac_u$ is a derivative: with $\overline G=B\Lambda^{r}$ one has $\Jac_u(P)=\overline G'$, so \cref{eq:concl-master} reads $\overline G'=\kappa\Lambda^{r-1}$. Comparing degrees, $\deg B+r\deg\Lambda-1=(r-1)\deg\Lambda$, that is $\deg B+\deg\Lambda=1$, which is impossible since $B$ lies in the maximal ideal and $\deg\Lambda\ge1$. This is the argument of \cref{thm:plane} in the coordinates of \cref{sec:higher}, and it closes the case because the degree of a derivative drops by exactly one: a curve admits an exact ramification count, and the count leaves no room.

In dimension three and above no such accounting holds. The quotient is a surface, $\Jac_{u,v}$ is not a derivative, and the discrepancy that the degree count would have to absorb is absorbed instead by the geometry: the quotient map is not proper, and over a general point of the discriminant the missing sheets escape to infinity rather than coming together (\cref{prop:disc}). The mechanism behind \cref{eq:map} needs somewhere for ramification to go, and a curve has nowhere. Consistently with this, \cref{rem:vasyunin} and \cref{prop:parabolic-hyp} show that it is the contracted locus, and not the hyperbolic signature, that makes a counterexample possible: when $\Lambda$ is constant, \cref{eq:concl-master} degenerates to the Keller condition for $\overline G$ and the graded problem in dimension $n$ becomes the ordinary Jacobian Conjecture in dimension $n-1$.

This also suggests where to look for structure we have not examined: the two GIT manifestations of the hyperbolic action, of which we have used only $\{x\neq0\}$, and whose other chamber appears related to the weighted projective line $\bP(1,2)$; we have not pursued this.

\subsection*{Where the classification stops}

\cref{tab:classify} records what the signature decides and what it does not. The elliptic case is empty in every dimension (\cref{thm:rigid}), dimension two is empty for every signature (\cref{thm:plane}), and the parabolic case is not a third phenomenon but a degeneration of the other two: with elliptic nonzero part it is the Jacobian Conjecture in the number of vanishing weights (\cref{thm:parabolic-elliptic}), and with hyperbolic nonzero part of the type \cref{eq:goodtype} it is the theory of \cref{sec:higher} verbatim (\cref{prop:parabolic-hyp}). What survives is a genuinely finer invariant than the signature: within the hyperbolic case the reduction to \cref{eq:concl-master} requires that the weight-one graded piece be free of rank one over the invariants, which holds exactly when there is a single positive weight and it equals $1$. Both failures are real. With one positive weight $a\ge2$ the invariant ring is not polynomial --- already $\bw=(2,-1,-1)$ gives a quadric cone --- and $\Jac_u$ loses its elementary meaning; with $n_+\ge2$, first possible in dimension four, the weight-one piece is generated by two elements and the factor $\Lambda$ has no analogue.

Four questions are left, in increasing order of how much we expect them to require.

\begin{enumerate}
\item Is $\Yparo[1,2,(4,3,1)]$ a single $\Gaut$-orbit (\cref{rem:orbit})? At these degrees \cref{prop:konestratum} forces the leading configuration of \cref{eq:map} and \cref{ex:enumerate} shows the leading data carries no continuous modulus, so what remains is the linear algebra of \cref{prop:determinantal}: a rank condition on a $13\times28$ matrix over an eleven-dimensional base. This is a computation, not a theorem, and we regard it as the first thing to do.
\item Is $\Kfam(3,(1,-1,-1))$ empty? By \cref{prop:oneone} the smallest undecided degrees are $\bd=(2,1,d_3)$ and $\bd=(1,2,d_3)$, and the same three-step search applies.
\item Does \cref{prop:konestratum} hold uniformly? Its proof --- normalise $\ell$, expand in the transverse variable, extract the top coefficient, integrate the resulting logarithmic identity, compare degrees --- appears not to use the particular degrees except at the last step. A statement that the stratum $k=1$ is empty for every $(r,s,\bd)$ would restore the content that \cref{cor:empty} does not have.
\item The two omitted hyperbolic shapes: positive weight $\ge2$, where the quotient is singular, and $n_+\ge2$, where it is smooth but the description \cref{eq:ntuple} fails. These need different methods, and the second is the first place where dimension four is not a formality.
\end{enumerate}

\subsection*{Arithmetic}

Restricting \cref{eq:map} to $\{x=0\}$ gives $(y,z)\mapsto(z+4y^{2},\,y,\,0)$, a Jonqui\`eres automorphism of the plane onto $\{c=0\}$, so every rational point of $\{c=0\}$ lifts. In the counting function
\[
  Z(X)=\#\{(a,b,c)\in\Q^{3}: H\le X,\ F^{-1}(a,b,c)\cap\Q^{3}\neq\emptyset\}
\]
this stratum should dominate, much as coordinate strata dominate the counting problems of \cite{sparsity}, while off the stratum the count is governed by the splitting of \cref{eq:cubic}; one expects the exponent there to come from a linear program, as in the positive-weight setting, though we have proved nothing in this direction. Part of the difficulty is foundational: the weighted heights of \cite{bgsh,sparsity} require positive weights, the stacky height frameworks require properness, and \cref{eq:map} lives where both hypotheses fail. A height theory for hyperbolic quotients --- a quotient height coupled with an orbit term --- would let the problem be posed precisely, and seems to us the natural first step.

The contrast with the parameter space is worth stating plainly, because it is the one place where the two settings meet inside a single object. By \cref{prop:wps} the triples $(A,B,\Lambda)$ modulo the natural $\Gm$ form a weighted projective space with all weights positive. The counting function $\Ncount(X)$ on it is therefore well posed, by the very theory of \cite{bgsh,sparsity} that $Z(X)$ falls outside: a hyperbolic counterexample is parametrised by an elliptic moduli space. Whether $\Ncount(X)$ measures anything beyond the growth of a single orbit is \cref{rem:orbit} again.

\subsection*{One classification}

The settings of \cite{sparsity} and of this note are two manifestations of one classification. For a $\Gm$-action the governing invariant is the position of $0$ relative to the weights. All weights of one sign, the elliptic case, gives proper, finite, ramified quotient maps, where the obstruction to lifting a rational point is a Kummer condition \cite{sparsity}; $0$ strictly between the weights, the hyperbolic case, gives non-proper quotients with contracted loci, where liftability is the splitting of a resolvent, in this example with nonabelian Galois group and thin image; and a zero weight interpolates, contributing a cylinder factor which, as \cref{thm:parabolic-elliptic} shows, carries the whole difficulty. In dimensions two and three every algebraic $\Gm$-action on affine space is linearizable \cite{gutwirth,kkmlr}, so there the classification covers all actions.

The Jacobian problem is trivial in the elliptic case (\cref{thm:rigid}) and in dimension two for every signature (\cref{thm:plane}), and the two facts have different proofs but the same source: in both, properness or an exact degree count forecloses the loss of a sheet. What is striking is that the dichotomy is not clean on the arithmetic side. \cref{rem:kummer} exhibits a Kummer condition inside the hyperbolic example, on the line where the stabilizer is $\muu{2}$, and \cref{rem:aut} suggests another at the level of moduli; the abelian obstruction of the positive-weight theory returns wherever the stabilizer is nontrivial. This points to a stratification by stabilizer type, refining the trichotomy rather than being refined by it, and it is the shape we would expect a general theory to take: a theory of lifting rational points along equivariant maps, organized by signature and by stabilizer, of which \cite{sparsity} is the elliptic case and the analysis above a first hyperbolic example.

For a torus of higher rank the geometric side is classical --- the invariant becomes the position of the origin relative to the convex hull of the weights, and the manifestations form the GIT fan --- so the arithmetic questions can at least be asked in that generality. We hope to return to this.

\end{document}